\documentclass[11pt,a4paper,reqno]{amsart}
\usepackage{amsmath,amsfonts,amssymb,amsthm,amscd,color}
\usepackage[english]{babel}
\usepackage[mathscr]{eucal}
\usepackage{graphicx}
\usepackage{hyperref,pdfsync}

\newcommand{\g}{\gamma}

\renewcommand{\d}{\delta}

\newcommand{\na}{\nabla}
\newcommand{\om}{\omega}
\newcommand{\OM}{\Omega}

\renewcommand{\t}{\tau}

\newcommand{\e}{\varepsilon}
\newcommand{\f}{\varphi}
\newcommand{\F}{\Phi}

\newcommand{\p}{\psi}
\renewcommand{\P}{\Psi}

\newcommand{\N}{{\mathbb N}}
\newcommand{\R}{{\mathbb R}}

\newcommand{\RR}{{\mathbb R}^2}

\newcommand{\Cc}{{\mathcal C}}

\newcommand{\Tc}{{\mathcal T}}

\newcommand{\curl}{{\rm curl}\,}

\renewcommand{\div}{{\rm div}\,}

\newcommand{\pd}{\partial}

\newcommand{\supp}{\operatorname{supp\,}}

\newcommand{\loc}{\operatorname{{loc}}}

\newtheorem{theorem}{Theorem}[section]
\newtheorem{proposition}[theorem]{Proposition}
\newtheorem{lemma}[theorem]{Lemma}

\newtheorem{definition}[theorem]{Definition}

\theoremstyle{remark}
\newtheorem{remark}[theorem]{Remark}

\numberwithin{equation}{section}

\begin{document}


\title[Topography influence on the Lake equations]{Topography influence on the Lake equations in bounded domains}
\author[C. Lacave, T. Nguyen, B. Pausader]{Christophe Lacave \and Toan T. Nguyen \and  Benoit Pausader }

\address[C. Lacave]{Universit\'e Paris-Diderot (Paris 7)\\
Institut de Math\'ematiques de Jussieu - Paris Rive Gauche\\
UMR 7586 - CNRS\\
B\^atiment Sophie Germain \\
Case 7012\\
75205 PARIS Cedex 13\\
France.} \email{lacave@math.jussieu.fr}

\address[T. Nguyen]{Department of Mathematics, Pennsylvania State University, University Park, State College, PA 16802.} \email{nguyen@math.psu.edu}

\address[B. Pausader]{Laboratoire Analyse, Géom\'etrie et Applications,
UMR 7539, Institut Galil\'ee, Universit\'e Paris 13, France.} \email{pausader@math.univ-paris13.fr}

\date{\today}

\begin{abstract}
We investigate the influence of the topography on the lake equations which describe the two-dimensional horizontal velocity of a three-dimensional incompressible flow.  We show that the lake equations are structurally stable under Hausdorff approximations of the fluid domain and $L^p$ perturbations of the depth. As a byproduct, we obtain the existence of a weak solution to the lake equations in the case of singular domains and rough bottoms. Our result thus extends earlier works by Bresch and M\'etivier treating the lake equations with a fixed topography and by G\'erard-Varet and Lacave treating the Euler equations in singular domains. 
\end{abstract}

\maketitle

\tableofcontents

\section{Introduction}\label{sec-Intro}
The lake equations are introduced in the physical literature as a two-dimensional geophysical model to describe the evolution of the vertically averaged horizontal component of the three-dimensional velocity of an incompressible Euler flow; see for example \cite{Greenspan,LOT,BM} and the references therein for physical discussions and derivation of the model. Precisely, the lake equations with prescribed initial and boundary conditions are
\begin{equation}
\label {lake-eqs} 
\left\lbrace \begin{aligned}
&\pd_t (b  v  ) + \div (b  v  \otimes v ) + b  \na p  =0 & \text{ for }(t,x)\in \R_+ \times \OM , \\
&\div (b  v ) =0 & \text{ for }(t,x)\in \R_+ \times \OM  ,\\
&(b  v )\cdot \nu =0 & \text{ for }(t,x)\in \R_+ \times \pd \OM , \\
&v (0,x) = v^0 (x) & \text{ for }x\in  \OM .
\end{aligned}
\right.
\end{equation}
Here $v=v(t,x)$ denotes the two-dimensional horizontal component of the fluid velocity, $p=p(t,x)$ the pressure, $b = b(x)$ the vertical depth which is assumed to be varying in $x$, $\OM \subset \R^2$ is the spatial bounded domain of the fluid surface, and $\nu$ denotes the inward-pointing unit normal vector on $\partial\OM$.

In case that $b$ is a constant, \eqref {lake-eqs}  simply becomes the well-known two-dimensional Euler equations, and the well-posedness is widely known since the work of Wolibner \cite{Wolibner} or Yudovich \cite{yudo}. When the depth $b$ varies but is bounded away from zero, the well-posedness is established in Levermore, Oliver and Titi \cite{LOT}. Most recently, Bresch and M\'etivier \cite{BM} extended the work in \cite{LOT} by allowing the varying depth to vanish on the boundary of the spatial domain. In this latter situation, the corresponding equations for the stream function are degenerate near the boundary and the elliptic techniques for degenerate equations are needed to obtain the well-posedness.

{\it In this paper, we are interested in stability and asymptotic behavior of the solutions to the above lake equations under perturbations of the fluid domain or rather perturbations of the geometry of the lake which is described by the pair $(\Omega,b)$.}
Our main result roughly asserts that the lake equations are persistent under these topography perturbations. That is, if we let $(\OM_n,b_n)$ be any sequence of lakes which converges to $(\OM,b)$ (in the sense of Definition \ref{def-ConvLake}), then the weak solutions to the lake equations on $(\OM_n,b_n)$ converge to the weak solution on the limiting lake $(\OM,b)$. In particular, we obtain strong convergence of velocity in $L^2$ and we allow the limiting domain $\OM$ to be very singular as long as it can be approximated by smooth domains $\OM_n$ in the Hausdorff sense. The depth $b$ is only assumed to be merely bounded. As a byproduct, {\it we establish the existence of global weak solutions of the equations \eqref{lake-eqs} for very rough lakes $(\OM,b)$. }

\bigskip

Let us make our assumptions on the lake more precise. We assume that the (limiting) lake $(\Omega,b)$ has a finite number of islands, namely: 
\begin{itemize}
\item[(H1)] $ \displaystyle \OM := \widetilde{\OM} \setminus \Bigl( \bigcup_{k=1}^N \Cc^k \Bigl)$,
where $\widetilde \OM$, $\Cc^k$ are bounded simply connected subsets of $\R^2$, $\widetilde \OM$ is open, and $\Cc^k$ are disjoints and compact subsets of $\widetilde \OM$.
\end{itemize}

\medskip

We assume that the boundary is the only place where the depth can vanish, namely:

\begin{itemize}
\item[(H2)]  There is a positive constant $M$ such that 
$$ \quad 0< b(x)\leq M \qquad\text{ in }  \OM.$$
In addition, for any compact set $K\subset \OM$ there exists positive numbers $\theta_K$ such that $ b(x) \ge \theta_K$ on $K$. 

\end{itemize}

\medskip

In the case of smooth lakes, we add another hypothesis. Near each piece of boundary, we allow the shore to be either of non-vanishing or vanishing topography with constant slopes in the following sense:

\begin{itemize}
\item[(H3)] There are  small neighborhoods  $\mathcal{O}^0$ and $\mathcal{O}^k$ of $\partial\widetilde{\OM}$ and $\partial\Cc^k$ respectively, such that, for $0\le k\le N$,
\begin{equation}\label{BForm}
b(x)=c(x)\left[d(x)\right]^{a_k} \qquad \text{ in } \mathcal{O}^k\cap \OM,
\end{equation}
where $c(x), d(x)$ are bounded $C^3$ functions in the neighborhood of the boundary, $c(x)\ge \theta>0$, $a_k \ge 0 $. Here the geometric function $d(x)$ satisfies $\Omega=\{d>0\}$ and $\nabla d\ne0$ on $\partial\Omega$.
\end{itemize}

In particular, around each obstacle $\mathcal{C}^k$, we have either {\em Non-vanishing topography} when $a_k=0$, in which case $b(x)\ge\theta$ or {\em Vanishing topography} if $a_k>0$ in which case $b(x)\to0$ as $x\to\partial\mathcal{C}^k$. As (H3) will be only considered for smooth lakes $\partial \OM \in C^3$, we note that up to a change of $c$, $\theta$,  we may take $d(x)=\hbox{dist}(x,\partial\OM)$.

\subsection{Weak formulations}
As in the case of the 2D Euler equations, it is crucial to use the notion of generalized vorticity,  which is defined by 
\begin{equation*}
\om := \frac1{b} \curl v = \frac{1}{b} (\pd_1 v_{2} - \pd_2 v_{1}). 
\end{equation*}
Indeed, taking the curl of the momentum equation, it follows that the vorticity formally verifies the following transport equation
\begin{equation}\label{transport}
 \pd_t (b \om) + \div(b v \om)=0.
\end{equation}
Thanks to the condition $\div(bv)=0$, we will show in Lemma \ref{lem-vorticity} that the $L^p$ norm of $b^{\frac1p}\omega$ is a conserved quantity for any $p\in [1,\infty]$, which provides an important estimate on the solution.

When $\Omega$ is not regular, the condition  $bv^0 \cdot \nu\vert_{\partial \Omega} = 0$ has to be understood in a weak sense:
\begin{equation}\label{imperm}
 \int_\Omega b(x) v^0(x) \cdot h(x)\, dx = 0, 
\end{equation}
for any test function $h$ in the function space $G(\Omega)$ defined by
\begin{equation*}
G (\Omega):=\Big \{w\in L^2(\Omega) \ : \ w=\nabla p, \ \text{ for some } p\in H^1_{\loc}(\Omega) \Big \}.
\end{equation*}
For $bv^0 \in L^2(\Omega)$, such a condition  is equivalent to
\begin{equation}\label{impermbis}
b v^0 \in \mathcal{H}(\Omega) ,
\end{equation}
where $\mathcal{H}(\Omega)$ denotes the completion of the function space $\{ \varphi \in C_c^\infty(\Omega) \ | \ \div \varphi =0 \}$ with respect to the usual $L^2$ norm. This equivalence can be found, for instance, in \cite[Lemma III.2.1]{Galdi}. Moreover, in \cite{Galdi} the author points out that if $\Omega$ is a regular bounded domain and if $bv^0$ is a sufficiently smooth function, then $bv^0$ verifies \eqref{imperm} if and only if $\div bv^0 = 0$ and $bv^0 \cdot \nu\vert_{\partial \Omega} = 0$.

Similarly to \eqref{imperm}, the weak form of the  divergence free and tangency conditions on $bv$  also reads:  
\begin{equation} \label{imperm2}
\forall h \in C_c^\infty\left([0,+\infty); G(\Omega)\right), \quad \int_{\R_+} \int_\Omega b(x) v(t,x)  \cdot  h(t,x) \, dxdt = 0.
\end{equation}

Next, we introduce several notions of global weak solutions to the lake equations. The first is in terms of the velocity.

\begin{definition}\label{defi-weak-velocity} Let $v^0$ be a vector field such that
\[ \div(b v^0)=0 \text{ in } \Omega, \quad bv^0\cdot \nu = 0 \text{ on } \partial\Omega, \text{ in the sense of \eqref{imperm}}\]
and
\[ \frac{\curl v^0}{b}\in L^\infty(\OM).\]
We say that $v$ is a global weak solution of the velocity formulation of the lake equations \eqref{lake-eqs}  with initial velocity $v^0$ if

i) $\frac{\curl v}{b} \in L^\infty (\R_+ \times \Omega)$ and $\sqrt{b}v \in L^{\infty}(\R_+; L^2(\Omega))$;

ii) $\div (bv)=0$ in $\Omega$ and $bv\cdot \nu=0$ on $\partial\Omega$ in the sense of \eqref{imperm2}; 

iii) the momentum equation in \eqref{lake-eqs} is verified in the distributional sense. That is, for all divergence-free vector test functions $\Phi \in C^\infty_c([0,\infty)\times \overline{\OM})$ tangent to the boundary, there holds that
\begin{equation}\label{eq-vel}
 \int_0^\infty\int_{\OM}\Phi_t \cdot v \, dxdt +\int_0^\infty \int_{\OM}(bv\otimes v): \nabla\Bigl(\frac{\Phi}b \Bigl) \, dxdt+\int_{\OM}\Phi(0,x) \cdot v^0(x)\, dx=0.
\end{equation}
\end{definition}

We emphasize that the test functions $\Phi$ are allowed to be in $C^\infty_c([0,\infty)\times \overline{\OM})$ rather than in $C^\infty_c([0,\infty)\times {\OM})$.  Namely, for any test functions $\Phi$ belonging to $C^\infty_c([0,\infty)\times \overline{\OM})$, there exists $T>0$ such that $\Phi \equiv 0$ for any $t>T$, and  such that $\Phi(t,\cdot) \in C^\infty(\overline{\OM})$ for any $t$, in the sense that $D^k \Phi(t,\cdot)$ is bounded and uniformly continuous on $\Omega$ for any $k\geq 0$ (see e.g. \cite{Galdi}).

In the above definition, it does not appear immediately clear how to make sense of \eqref{eq-vel} for test functions supported up to the boundary due to the term $\Phi/ b$ which would then blow up at the boundary. For this reason, let us introduce a weak {\it interior} solution $v$ of the velocity formulation to be the weak solution $v$ as in Definition \ref{defi-weak-velocity} with the test functions $\Phi$ in \eqref{eq-vel} being supported {\it inside} the domain, i.e. $\Phi\in C^\infty_c([0,\infty)\times\Omega)$. For this weaker solution, \eqref{eq-vel} then makes sense under the regularity (i) when $b\in W^{1,\infty}_{\loc}(\Omega)$ (because (H2) gives an estimate of $b^{-1}$ locally in space). Later on in  Appendix \ref{sect equiv defi}, we show that \eqref{eq-vel} indeed makes sense with the test functions supported up to the boundary when the lake is smooth, even in the case of vanishing topography.

\medskip

The second formulation of weak solutions is in terms of the vorticity and reads as follows.

\begin{definition}\label{defi-weak-sol}Let $(v^0,\om^0)$ be a pair such that
\begin{equation}\label{defi-a}
\div(b v^0)=0 \text{ in } \Omega, \quad bv^0\cdot \nu = 0 \text{ on } \partial\Omega \qquad \text{ (in the sense of \eqref{imperm}}) 
\end{equation}
and
\begin{equation}\label{defi-b}
\omega^0\in L^\infty(\OM), \quad \curl v^0 = b \om^0 \qquad \text{(in the distributional sense)}.
\end{equation}
We say that $(v,\omega)$ is a global weak solution of the vorticity formulation of the lake equations on $(\Omega,b)$  with initial condition $(v^0,\omega^0)$ if

i) $\omega \in L^\infty (\R_+ \times \Omega)$ and $\sqrt{b}v \in L^{\infty}(\R_+; L^2(\Omega))$;

ii) $\div (bv)=0$ in $\Omega$ and $bv\cdot \nu=0$ on $\partial\Omega$ in the sense of  \eqref{imperm2};

iii) $\curl v = b \omega$ in the distributional sense;

iv) the transport equation \eqref{transport} is verified in the sense of distribution. That is, for all test functions $\varphi \in C^\infty_c([0,\infty)\times \overline{\OM})$ such that $\partial_{\tau} \varphi\vert_{\partial\OM}\equiv0$ (i.e. constant on each piece of boundary), there holds that
\begin{equation}\label{eq-vort}
 \int_0^\infty\int_{\OM}\f_t b \om \, dxdt +\int_0^\infty \int_{\OM}\na\f \cdot v b \om \, dxdt+\int_{\OM}\f(0,x)b \om^0(x)\, dx=0.
\end{equation}
\end{definition}

We also introduce a weaker intermediate notion:  weak {\it interior} solution of the vorticity formulation to be the weak solution $(v,\omega)$ as in Definition \ref{defi-weak-sol} with the test functions being  supported {\it inside} the domain: i.e. $\varphi\in C^\infty_c([0,\infty)\times\Omega)$.

We will establish the relations between these definitions in Appendix \ref{sect equiv defi}. For example, when the lake is smooth, all velocity and vorticity formulations are equivalent. 

Following the proof of Yudovich \cite{yudo}, Levermore, Oliver and Titi \cite{LOT} established existence and uniqueness of a global weak solution (with the vorticity formulation) in the case of non-vanishing topography, assuming the lake is smooth and simply connected.  Recently, Bresch and M\'etivier \cite{BM} extended the well-posedness to the case of vanishing topography. In both of these works, $\Omega$ is assumed to be simply connected, $\partial \Omega \in C^3$, and $b\in C^3(\overline{\Omega})$. The essential tool in establishing the well-posedness is a Calder\'on-Zygmund type inequality. This inequality is highly non trivial to obtain if the depth vanishes, and the proof requires to work with degenerate elliptic equations.

In Section \ref{sect WP}, we shall sketch the proof of the well-posedness of the lake equations under our current setting (H1)-(H3):
\begin{theorem}\label{theo-WP} Let $(\Omega,b)$ be a lake verifying Assumptions (H1)-(H3) and $(\partial \OM,b)\in C^3\times C^3(\overline{\Omega})$. Then for any pair $(v^{0},\om^0)$ such that $b^{-1} \curl v^0=\om^0\in L^{\infty}(\OM)$, there exists a unique global weak solution $(v,\omega)$ to the lake equations that verifies both the velocity and vorticity formulations.  Furthermore, we have that
\[\omega \in C(\R_+,L^r(\Omega)),\quad v \in C(\R_+,W^{1,r}(\Omega)), \quad v\cdot \nu =0 \text{ on }\partial\Omega,\]
for arbitrary $r$ in $  [1,\infty)$ and the circulations of $v$ around $\mathcal{C}^k$ are conserved for any $k=1\dots N$.
\end{theorem}

In fact, when the domain is not simply connected, the vorticity alone is not sufficient to determine the velocity uniquely from \eqref{defi-a}-\eqref{defi-b}. We will then introduce in Section \ref{section 2.1} the weak circulation for lake equations, derive the Biot-Savart law (the law which yields the velocity in term of the vorticity and circulations), and prove the Kelvin's theorem concerning conservation of the circulation.

\subsection{Assumptions}
 For each $n\ge 1$, let $(\Omega_n,b_n)$ be a lake of either vanishing or non-vanishing or mixed-type topography as described above in (H1)-(H3) with constants $\theta_n,M_n,a_{0,n},\dots,a_{N,n}$ and function $d_n(x)$.
 
In what follows, we write $(\Omega_{0},b_{0}) = (\Omega,b)$, which will play the role of the limiting lake. We assume that these lakes have the same finite number of islands $N$, namely for any $n\geq 0$
\[ \OM_n := \widetilde{\OM}_n \setminus \Bigl( \bigcup_{k=1}^N \Cc_n^k \Bigl) ,\]
where $\widetilde \OM_n, \Cc_n^k$ are simply connected subsets of $\R^2$, $\widetilde \OM_n$ is open, and $\Cc_n^k\subset\widetilde{\OM}_n$ are disjoint and compact. In addition, let $D$ be a big enough subset so that $\OM_n \subset D$, $n\ge 0$.  

\begin{definition}\label{def-ConvLake} Assume that $(\partial \OM_n,b_n)\in C^3\times C^3(\overline{\Omega_{n}})$ for all $n\ge 1$.
We say that the sequence of lakes $(\OM_n,b_n)$ converges to the lake  $(\OM,b)$ as $n \to \infty$ if there hold
\begin{itemize}
\item $\widetilde \OM_n \to \widetilde \OM$ in the Hausdorff sense;
\item $\Cc_n^k \to \Cc^k$ in the Hausdorff sense;
\item $b_n$ is uniformly bounded in $L^\infty(\OM)$ and for any compact set $K\subset \Omega$ there exist positive $\theta_{K}$ and sufficiently large $n_0(K)$ such that $b_{n}(x)\geq \theta_{K}$ for all $x\in K$ and $n\ge n_0(K)$,
\item $b_n \to b$ in\footnote{Note that since $b_n$ is uniformly bounded in $L^\infty$, we directly see that the convergence holds in $L^p$, $p<\infty$.} $L^{1}_{\loc} (\OM)$.
\end{itemize}
\end{definition}

Here $\Omega_n$ converges to $\Omega$ in the Hausdorff sense if and only if the Hausdorff distance between $\Omega_n$ and $\Omega$ converges to zero. See for example \cite[Appendix B]{GV_lac} for more details about the Hausdorff topology, in particular the Hausdorff convergence implies the following proposition: for any compact set $K\subset \Omega$, there exists $n_{K}>0$ such that $K\subset \Omega_{n}$ for all $n\geq n_{K}$, which gives sense to the fourth item of the above definition.

Definition \ref{def-ConvLake}, allows in particular the limit $a_n\to a_0 = 0$, with $a_n$ introduced as in (H3). This means that the passage from a lake of the vanishing type in which the slope gets steeper and steeper to a lake of non-vanishing type is allowed. This appears to be complicated to deduce from the analysis in \cite{BM}, where the condition $a_0>0$ is crucial. Remarkably, it turns out that uniform estimates of the velocity in $W^{1,p}$ are not needed in order to pass to the limit. As will be shown, $L^2$ estimates are sufficient.

\subsection{Main results}

As mentioned, a velocity field is uniquely determined by its vorticity and its circulation around each obstacle. We recall that when the velocity field $v$ is continuous, the circulation around each obstacle $\mathcal{C}^k$ is classically defined by 
$$\gamma^k_{\mathrm{cl}} : = \oint_{\partial \mathcal{C}^k} v\cdot d{\bf s}.$$ 
However, with a low regularity velocity field as in our definitions of weak solutions, such a path integral might not be well defined a priori. We are led to introduce the generalized circulation 
$$\gamma^k(v) : = \int_\Omega \div (\chi^k  v^\perp ) \; dx $$  
where $\chi^k$ is some smooth cut-off function that is equal to one in a neighborhood of $\mathcal{C}^k$ and zero far away from $\mathcal{C}^k$. Observing that $\div (\chi^k  v^\perp ) = - \nabla^\perp \chi^k  \cdot v - \chi^k \curl v$, the generalized circulation is well defined for the weak solution $v$ by condition (i) in Definitions \ref{defi-weak-velocity} and \ref{defi-weak-sol} (indeed, (H2) implies that $v$ belongs to $L^2(\supp \nabla^\perp \chi^k)$). Later in Section \ref{sect WP}, we will show that such a generalized circulation enjoys the same property as that of the classical one $\gamma^k_\mathrm{cl}$. Most importantly, the velocity field is uniquely determined by the vorticity and the circulations; see Section \ref{sect WP}. 

Our assumptions on the convergence of the initial data are in terms of the vorticity and circulations. Precisely, we assume that the initial vorticity $\omega^0_n$ is uniformly bounded:
\begin{equation}\label{ConvHyp1}
\| \om^0_n \|_{L^\infty(\OM_n)}\leq M_0,
\end{equation} for some positive $M_0$, and there holds the convergence 
\begin{equation}\label{ConvHyp1b}
 \omega^0_n \rightharpoonup \omega^0\text{ weakly in }L^1(D),
\end{equation}
as $n \to \infty$. Here $\omega_n^0$ is extended to be zero in $D\setminus \OM_n$. Concerning the circulations, we assume that the sequence $\gamma_n=\{\gamma^k_n\}_{1\le k\le N}\in\mathbb{R}^N$ converges to a given vector $\gamma=\{\gamma^k\}_{1\le k\le N}$ in the sense that
\begin{equation}\label{ConvHyp2}
\sum_{k=1}^N\vert \gamma^k_n-\gamma^k\vert\to 0,
\end{equation} as $n \to \infty$. 
Then, for each $n\ge 1$, we define the initial velocity field $v^0_n$ to be the unique solution of the following elliptic problem in $\OM_n$:
\begin{equation}\label{v0n}
\div (b_{n} v_{n}^0)=0, \quad (b_{n} v_{n}^0)\cdot \nu\vert_{\pd \OM_{n}}=0  , \quad \curl v_{n}^0 = b_{n} \om^0_n,\quad \gamma^k_{n}(v_{n}^0)=\gamma^k_{n} \ \forall 1\le k\le N.
\end{equation}
The existence and uniqueness of $v^0_n$ are established in Section \ref{sect WP}. 

Our first main theorem is concerned with the stability of the lake equations:

\begin{theorem}\label{theo-main}
Let $(\Omega,b)$ be a lake satisfying Assumptions (H1)-(H3) with $(\partial \OM,b)\in C^3\times C^3(\overline{\Omega})$. Assume that  there is a sequence of lakes $(\Omega_n,b_n)$ which converges to $(\Omega,b)$ in the sense of Definition \ref{def-ConvLake}. Assume also that $(\omega^0_n,\gamma_n,v^0_n)$ are as in \eqref{ConvHyp1}--\eqref{v0n}.
Let $(v_n,\omega_n)$ be the unique weak solution of the lake equations \eqref {lake-eqs}  on the lake $(\OM_n, b_n)$ with initial velocity $v_{n}^0$, $n \ge 1$. Then, there exists a pair $(v,\omega)$ so that 
\[ v_n \to v \text{ strongly in } L^2_{\loc}(\R_+; L^2(D)),\qquad  \om_n\rightharpoonup  \om \text{ weak-$*$ in } L^\infty(\R_+\times D) .\]
Furthermore, $(v,\omega)$ is the unique weak solution of the lake equations on the lake $(\OM, b)$ with initial vorticity $\omega^0$ and initial circulation $\gamma \in \R^N$.
\end{theorem}

This theorem, whose proof will be given in Section \ref{sect 3}, links together various results on the lake equations, namely the flat bottom case (Euler equations \cite{yudo}), non-vanishing topography \cite{LOT} and vanishing topography \cite{BM}. Indeed, we allow the limit $a_{n}\to 0$ (passing from vanishing topography to  non vanishing topography), or the limit $\theta_{n}\to 0$ if $b_n=b+\theta_n$ where $b$ verifies (H2)-(H3) (passing from non vanishing topography to vanishing topography). The convergence of the solutions of the Euler equations when the domains converge in the Hausdorff topology is a recent result established by G\'erard-Varet and Lacave \cite{GV_lac}, based on the $\gamma$-convergence on open sets (a brief overview of this notion is given in Appendix \ref{app_gammaconv}). The present paper can be regarded as a natural extension of \cite{GV_lac} to the lake equations i.e. to the case of non-flat bottoms $b_n$ when we consider a weak notion of convergence of $b_n$.

The $\gamma$-convergence is an $H_{0}^1$ theory on the stream function (or an $L^2$ theory on the velocity). Bresch and M\'etivier have obtained estimates in $W^{2,p}$ for any $2\le p<+\infty$ (namely, the Calder\'on-Zygmund inequality) for the stream function, which is necessary for the uniqueness problem or to give a sense to the velocity formulation. For our interest in the sequential stability of the lake solutions, it turns out that we can treat our problem without having to derive uniform estimates in $W^{2,p}$, which appear hard to obtain. In fact, we will first prove the convergence of a subsequence of $v_n$ to $v$ and show that the limiting function $v$ is indeed a solution of the limiting lake equations. Since the Calder\'on-Zygmund inequality is verified for the solution of the limiting lake equations, the uniqueness yields that the whole sequence indeed converges to the unique solution in $(\OM,b)$.

More importantly, since the Calder\'on-Zygmund inequality is not used in the compactness argument, it follows that the existence of a weak solution to the lake equations with non-smooth domains or non-smooth topography can be obtained as a limit of solutions to the lake equations with smooth domains. Our second main theorem is concerned with non-smooth lakes which do not necessarily verify (H3).

\begin{theorem}\label{theo-non-smooth}
Let $(\Omega,b)$ be a lake satisfying (H1)-(H2). We assume that for every $1\le k\le N$, $\Cc^k$ has a positive Sobolev $H^1$ capacity. For any $\omega^0\in L^\infty(\Omega)$ and $\gamma\in \R^N$, there exists a global weak solution $(v,\omega)$ of the lake equations in the vorticity formulation on the lake $(\OM, b)$ with initial vorticity $\omega^0$ and initial circulation $\gamma \in \R^N$. This solution enjoys a Biot-Savart decomposition and its circulations are conserved in time. If we assume in addition that $b\in W^{1,\infty}_{\loc}(\Omega)$  then $(v,\omega)$ is also a global weak {\em interior} solution in the velocity formulation.
\end{theorem}

Let us mention that we do not assume any regularity of $\partial \Omega$; for instance, $\partial \OM$ can be the Koch snowflake.  To obtain solutions for the vorticity formulation, we do not need any regularity on $b$ either; it might not even be continuous. But even in the case where we assume the bottom to be locally lipschitz, choosing $b_n:=b+\frac1n$ we can consider a zero slope: $b(x)=e^{-1/d(x)}$ or non constant: $b(x)=d(x)^{a(x)}$; our theorem states that $(v,\omega)$ is a solution of the vorticity formulation and an {\it interior} solution of the velocity formulation. Such a result might appear surprising, because the known existence result requires that the lake domain is smooth, namely $(\partial \OM,b)\in C^3\times C^3(\overline{\Omega})$ and (H3).

The Sobolev $H^1$ capacity of a compact set $E \subset \R^2$ is defined by 
$$ {\rm cap}(E) \: := \: \inf \{ \| v \|^2_{H^1(\R^2)}, \: v \ge 1 \: \mbox{ a.e.    in a neighborhood of } E\},  
$$
with the convention that ${\rm cap}(E)= +\infty$ when the set in the r.h.s. is empty. We refer to  \cite{henrot} for an extensive study of this notion (the basic properties are listed in \cite[Appendix A]{GV_lac}, in particular we recall that a material point has a zero capacity whereas the capacity of a Jordan arc is positive).

Apparently, in such non-smooth lake domains, the Calder\'on-Zygmund inequality is no longer valid, and hence the well-posedness is delicate.  For existence, our construction of the solution follows by approximating the non-smooth lake by an increasing sequence of smooth domains in which the solutions are given from Theorem \ref{theo-WP}.

Finally, we leave out the question of uniqueness in the case of non-smooth lakes. We refer to \cite{lac-uni} for a uniqueness result for the 2D Euler equations in simply-connected domains with corners. In \cite{lac-uni} the velocity is shown in general not to belong to $W^{1,p}$ for all $p$ (precisely, if there is a corner of angle $\alpha>\pi$, then the velocity is no longer bounded in $L^p\cap W^{1,q}$, $p>p_\alpha$, $q>q_\alpha$ with $p_\alpha\to 4$ and $q_\alpha\to 4/3$ as $\alpha\to 2\pi$).

\section{Well-posedness of the lake equations for smooth lake}\label{sect WP}

In this section, we sketch the proof of existence of the lake equations in a non-simply connected domain (Theorem \ref{theo-WP}). The proof can be outlined as follows: 
\begin{itemize}
 \item we first prove existence of a global weak interior solution in the vorticity formulation. The proof follows by adding an artificial viscosity (as was done in \cite{LOT1}) and obtaining compactness for the vanishing viscosity problem (Section \ref{sect 2.2});
 \item as the lake is smooth, we then use the Calder\'on-Zygmund inequality established in \cite{BM}, which in turn implies that for arbitrary $r\ge 1$, $\omega \in C(\R_+,L^r(\Omega))$, $v \in C(\R_+,W^{1,r}(\Omega))$, and $v\cdot \nu =0$ on $\partial\Omega$;
 \item thanks to the regularity close to the boundary, we can show  by a continuity argument that \eqref{eq-vort} is indeed verified for test functions supported up to the boundary (Proposition \ref{prop A4}). The existence of a global weak solution in the vorticity formulation (with conserved circulations) is then established. The solution also verifies the velocity formulation due to the equivalence of the two formulations (Proposition \ref{prop A3}). 
 \item finally, uniqueness of a global weak solution is shown in Section \ref{sec-uniqueness} by following the celebrated method of Yudovich.
\end{itemize}

Essentially, this outline of the proof was introduced by Yudovich in his study of two-dimensional Euler equations \cite{yudo}, and it was used in \cite{LOT,BM} in the case of the lake equations. We shall provide the proof with more details as it will be crucial in our convergence proof later on.

Throughout this section, we fix a smooth lake $(\Omega,b)$ namely:
\begin{equation}\label{smooth lake}
(\Omega,b) \text{ satisfying Assumptions (H1)-(H3) (see Section \ref{sec-Intro}) and }  (\partial \OM,b)\in C^3\times C^3(\overline{\Omega}). 
\end{equation}
We allow the lake to have either vanishing or non-vanishing topography. We shall begin the section by deriving the Biot-Savart law. We then obtain the well-posedness of the lake equations \eqref{lake-eqs} in the sense of Definitions  \ref{defi-weak-velocity} and \ref{defi-weak-sol}.

\subsection{Auxiliary elliptic problems} \label{section 2.1}
Let us introduce the function space 
\[ X:= \Bigl\{ f\in H^1_0(\Omega)~:~ b^{-1/2} \na f \in L^2(\Omega) \Bigl\}.\]
We will sometimes write the function space as $X_b$ instead of $X$ to emphasize the dependence on $b$.
Clearly,  $(X,\|\cdot\|_X)$ is a Hilbert space with inner product $\langle f,g\rangle_X : = \langle b^{-1/2} \nabla f, b^{-1/2} \nabla g\rangle_{L^2}$ and norm $\|f\|_{X} := \langle f,f\rangle_X^{1/2}$.  
Our first remark is concerned with the density of $C^\infty_c(\OM)$ in $X$.
\begin{lemma}\label{lem density} Let $(\Omega,b)$ be a smooth lake in the sense of \eqref{smooth lake}. Then 
$C^\infty_c(\OM)$ is dense in $X$ with respect to the norm $\|\cdot \|_X$.
\end{lemma}

The proof relies on a variant of the Hardy's inequality. As it was noted in the introduction, we can consider that $d$ is the distance to the boundary in \eqref{BForm}. With the notation: 
\begin{equation}\label{Neighborhood} \partial\Omega_R:=\{x\in\Omega:0\le d(x)\le R\},
\end{equation}
we establish the following Hardy type inequality:

\begin{lemma}\label{HardyLem}
Let $(\Omega,b)$ be a smooth lake in the sense of \eqref{smooth lake}. Then the following inequality holds uniformly for every $f\in H^1_0(\Omega)$ and any positive $R$:
\begin{equation}\label{Hardy}
\Vert b^{-1/2}(f/d)\Vert_{L^2(\partial\Omega_R)} \lesssim \Vert b^{-1/2}\nabla f\Vert_{L^2(\partial\Omega_R)}.
\end{equation}
\end{lemma}

Here in Lemma \ref{HardyLem} and throughout the paper, the notation $g\lesssim h$ is used to mean a uniform bound $g \le C h$, for some universal constant $C$ that is independent of the underlying parameter (in \eqref{Hardy}, small $R>0$ and $f$).

\begin{proof}[Proof of Lemma \ref{HardyLem}]
We start with the following claim:  for any $f\in H^1_0(\Omega)$ and any positive $R$, there holds that
\begin{equation}\label{Hardy2}
\int_{R\le d(x)\le 2R} \vert f(x)\vert^2\, dx\lesssim R^2 \int_{ d(x)\le 2R}\vert\nabla f(x)\vert^2\, dx . 
\end{equation}
The claim follows directly from the fundamental theorem of Calculus and the standard H\"older's inequality at least for smooth compactly supported functions. By density, it extends to $H^1_0(\Omega)$.

Next, by \eqref{Hardy2}, the lemma follows easily for functions $f\in H_0^1(\Omega)$ whose support is away from the set $\bigcup_{k:a_k>0}\mathcal{O}_k$. It suffices to consider functions $f$ that are supported in the set $\mathcal{O}_j$ for $a_j>0$. Again by \eqref{Hardy2}, we can write 
\begin{equation*}
\begin{split}
\Big \Vert b^{-1/2} (f/d) \Big \Vert_{L^2(\partial\Omega_R)}^2\quad &=\quad \sum_{k\in\mathbb{N}^*}\int_{R\le 2^kd(x)\le 2R}\left(\frac{f(x)}{d(x)}\right)^2\frac{dx}{b(x)}\\
\quad&\lesssim \quad \sum_{k\in\mathbb{N}^*}(R2^{-k})^{-(a_j+2)}\int_{R\le 2^kd(x)\le 2R} \vert f(x)\vert^2\, dx\\
\quad&\lesssim \quad  \sum_{k\in\mathbb{N}^*}(R2^{-k})^{-a_j}\int_{2^kd(x)\le 2R} \vert \nabla f(x)\vert^2\, dx\\
\quad&\lesssim \quad \int_{\partial\Omega_R}\left(\sum_{k\in\mathbb{N}^*:\,\, 2^kd(x)\le 2R} (R2^{-k})^{-a_j}\right)\vert \nabla f(x)\vert^2\, dx.\\
\end{split}
\end{equation*}
Since the summation  in the parentheses in the last line above is bounded by $b^{-1}$, the integral on the righthand side is bounded by $\Vert b^{-1/2}\nabla f\Vert_{L^2(\partial\Omega_R)}^2$. The lemma is thus proved. \end{proof}

\begin{proof}[Proof of Lemma \ref{lem density}] Fix $\varepsilon>0$ and $f \in X$. It suffices to construct a cut-off function $\chi\in C^1_c(\Omega)$ such that
\begin{equation}\label{CutOff}
\Vert (1-\chi)f\Vert_{X}\le\varepsilon.
\end{equation}
The lemma would then follow simply by approximating the compactly supported function $\chi f$ with its $C^\infty_c$ mollifier functions. 

Now since $f\in X$, there exists a positive $R_\epsilon$ such that
\begin{equation}\label{CutOff2}
\int_{\partial\Omega_{R_\epsilon}}\vert\nabla f(x)\vert^2\frac{dx}{b(x)}\le\varepsilon^2.
\end{equation}
Let us introduce a cut-off function $\eta\in C^\infty(\mathbb{R}_+)$ such that $0\le\eta\le 1$, $\eta(z)\equiv1$ if $z\ge 1$ and $\eta(z)\equiv 0$ if $z\le 1/2$ and define
\begin{equation*}
\chi(x)=\eta(d(x)/(R_\epsilon)).
\end{equation*}
Clearly, $\chi\in C^1_c(\Omega)$. In addition, we note that $\nabla[(1-\chi)f]=(1-\chi)\nabla f-f\nabla\chi$. It then follows by \eqref{CutOff2}  that
\begin{equation*}
\int_{\Omega}(1-\chi(x))^2\vert\nabla f(x)\vert^2\frac{dx}{b(x)}
\le \int_{\partial\Omega_{R_\epsilon}}\vert\nabla f(x)\vert^2\frac{dx}{b(x)}
\le \varepsilon^2.
\end{equation*}
Meanwhile using the fact that
\begin{equation*}
\vert f\nabla\chi\vert=\vert  R_\epsilon^{-1} f\eta^\prime(d(x)/R_\epsilon) \nabla d(x)\vert\le |(f/d)(x)| \Vert\eta^\prime\Vert_{L^\infty}
\end{equation*}
and Lemma \ref{HardyLem}, we obtain
\begin{equation*}
\begin{split}
\int_{\Omega}\vert f(x)\nabla\chi(x)\vert^2\frac{dx}{b(x)}\le \Vert\eta^\prime\Vert_{L^\infty}\int_{\partial\Omega_{R_\epsilon}}\frac{\vert f(x)\vert^2}{d(x)^2}\frac{dx}{b(x)}\lesssim \int_{\partial\Omega_{R_\epsilon}} \vert\nabla f(x)\vert^2\frac{dx}{b(x)}\lesssim\varepsilon^2.
\end{split}
\end{equation*}
This yields \eqref{CutOff} which completes the proof of the lemma.\end{proof}

\medskip

Next, we consider the following auxiliary elliptic problem
\begin{equation}\label{SEPb}
\div \Bigl[ \frac1b \na \p \Bigl] = f \text{ in } \OM, \quad  \text{with}\quad \p_{\vert_{\partial \OM}}  = 0.
\end{equation}
\begin{proposition}\label{prop elliptic}Let $(\Omega,b)$ be a smooth lake in the sense of \eqref{smooth lake}.
Given $f\in L^2(\OM)$, there exists a unique (distributional) solution $\psi \in X$ of the problem \eqref{SEPb}. 
\end{proposition}

\begin{proof} Let us introduce the functional
\begin{equation*}
E(\p) : = \int_{\OM} \Bigl( \frac1{2b} |\na \p|^2 + f\p \Bigl)\, dx.
\end{equation*}
Since $f \in L^2$, the functional $E(\cdot)$ is well-defined on $X$. Let $\p_k\in X $ be a minimizing sequence. Thanks to the Poincar\'e inequality and the fact that $b$ is bounded,  $\psi_k$ is uniformly bounded in $X$. Up to a subsequence, we assume that $\p_k\rightharpoonup \p$ weakly in $X$. By the lower semi-continuity of the norm, we obtain that
\[ E(\psi) = E(\liminf_{k\to \infty} \p_k) \leq \liminf_{k\to \infty} E(\p_k).\]
Hence, $\p\in X$ is indeed a minimizer. In addition, by minimization, the first variation of $E(\psi)$ reads 
\begin{equation}\label{SEPb-dist}
  \int_{\OM}  \Bigl(\frac1b \nabla\f \cdot \na \p+ \f f\Bigl) \, dx =0, \qquad \ \forall \f\in C^\infty_c(\OM),
\end{equation}
which shows that $\psi$ is a solution of \eqref{SEPb}. We recall that the Dirichlet boundary condition is encoded in the function space $X$. For the uniqueness, let us assume that $\p\in X$ is a solution with $f\equiv 0$. Then, \eqref{SEPb-dist} simply reads $\langle \varphi, \psi\rangle_X = 0$, for arbitrary $\f \in C^\infty_c(\OM)$. It follows by density (see Lemma \ref{lem density}) that $\|\psi\|_X =0$ and so $\psi =0$. This proves the uniqueness as claimed. 
\end{proof}

\begin{definition}\label{defi simili}
We say that $\F$ is a simili harmonic function if 
\[\F\in H^1_0(\widetilde \OM), \quad  b^{-1/2} \nabla \F \in L^2(\OM),\]
where $\widetilde \OM$ is as introduced in (H1), so that $\Phi$ solves the problem
\[ \div \Bigl[ \frac1b \na \F\Bigl] = 0 \text{  in  }\Omega, \quad \text{ and } \pd_{\t} \F =0 \text{ on } \pd \OM.\]
We denote by $\mathcal{SH}$ the space of simili harmonic functions.
\end{definition}

We remark that since a simili harmonic function $\F$ belongs to $H^1(\Omega)$, we can define its trace at the boundary, and so $\pd_{\t} \F =0$ should be understood as its trace being constant on each connected component of $\partial \OM$.

\begin{proposition}\label{prop basis} Let $(\Omega,b)$ be a smooth lake in the sense of \eqref{smooth lake}.
For $1\le k\le N$, there exists a unique simili harmonic function $\f^k$ such that
\[\f^k=0 \text{ on } \pd \widetilde{\OM}, \quad \f^k = \d_{ik} \text{ on }\pd \Cc^i, \ \forall i=1\dots N.\]
Moreover, the family $\{ \f^k \}_{k=1..N}$ forms a basis for the set of simili harmonic functions.
\end{proposition}

\begin{proof}
Let $\delta=\frac{1}{10}\min_{i\ne j}\{\hbox{dist}(\Cc^i,\Cc^j),\hbox{dist}(\Cc^i,\partial \widetilde\OM)\}$. For each $k$, we introduce a cut-off function $\chi^k\in C^\infty_c(\widetilde{\OM})$ which is supported in a $\delta$-neighborhood of $\Cc^k$ and satisfies
\begin{equation}\label{chi}
\chi^k(x)=0 \text{ if } d(x,\Cc^k)>\delta,\quad \chi^k(x)=1 \text{ if } d(x,\Cc^k)<\delta/2.
\end{equation}
In particular,
\[
\chi^k = \d_{ik} \text{ in a neighborhood of }\Cc^i, \ \forall i=1\dots N.
\]
By Proposition \ref{prop elliptic}, there exists a unique solution  $\tilde\f^k\in X$  to the problem
 \[ \div \Bigl[ \frac1b \na \tilde \f^k \Bigl] = - \div \Bigl[ \frac1b \na \chi^k \Bigl]  \text{ in } \OM, \quad  \tilde \f^k =0 \text{ on } \pd \OM.\]
Indeed, since $\na \chi^k$ is smooth and vanishes near the boundaries, the right-hand side of the above problem clearly belongs to $L^2(\OM)$. Now if we define
\begin{equation}\label{defi phi}
 \f^k:= \tilde \f^k + \chi^k,
\end{equation}
the existence of a simili harmonic function $\varphi^k$ follows at once as claimed. 

The uniqueness follows from the uniqueness result in Proposition \ref{prop elliptic}: indeed, let $\f^1$ and $\f^2$ be two simili harmonic functions which have the same trace on each component of $\partial \OM$. Then, $\F:=\f^1-\f^2$ belongs to $H^1_0(\OM)$ and so $\F\in X$, which is the function space where the uniqueness was proved.

Finally, since any simili-harmonic function by definition is constant on each connected component of $\partial\OM$, it follows clearly that the family $\{\f^k\}_{1\le k\le N}$ forms a basis of $\mathcal{SH}$.
\end{proof}

To recognize the divergence free condition \eqref{imperm}, we need the following simple lemma:
\begin{lemma}\label{lem div}
Let $(\Omega,b)$ be a lake satisfying Assumption (H1)-(H2) (not necessarily smooth).
Let $\psi\in X$, $c_k\in \R$ and $\chi^k\in C^\infty(\OM)$ as introduced in \eqref{chi}. Then the vector function 
$$v:=\frac{\nabla^\perp(\p+\sum_{k=1}^N c_k \chi^k)}b$$
satisfies 
\begin{equation}\label{div-cond} \div(b v)=0 \text{ in } \Omega, \quad bv\cdot \nu = 0 \text{ on } \partial\Omega \quad \text{(in the sense of \eqref{imperm})}.\end{equation}
Conversely, let $v$ be a vector field so that $bv\in L^2(\OM)$ and \eqref{div-cond} holds. Then there exists $\psi \in H^1_0(\widetilde \OM)$ such that
\[ bv = \nabla^\perp \psi \text{ in $\overline \Omega$ and }  \pd_{\t} \p =0 \text{ on } \pd \OM.\]
\end{lemma}
\begin{proof}
As $ \p\in X\subset H_{0}^1(\Omega)$, we can easily check that $\nabla^\perp \p$ belongs to $\mathcal{H}(\Omega)$ (see \eqref{impermbis}). Moreover, since $\chi^k$ is smooth and constant in a neighborhood of the boundary, $\nabla^\perp \chi^k$ verifies the boundary condition \eqref{imperm} and so does \eqref{div-cond}. 

The second one is a classical statement which does not depend on the regularity of $\partial \OM$. Indeed, as $bv$ verifies \eqref{impermbis}, we can find a  divergence-free vector $v_n\in C^\infty_c(\OM)$, such that $v_n \to bv$ in $L^2(\OM)$. Then $v_n$ is supported in a smooth set, and we can use the classical Hodge-De Rham theorem: $v_n=\nabla^\perp \p_n$ where $\p_n$ is constant near the boundary. Choosing $\p_n$ such that $\p_n(x)\equiv 0$ in a neighborhood of $\partial \widetilde \OM$, we then infer by Poincar\'e inequality that $\p_n \to \p$ strongly in $H^1(\OM)$, hence $bv=\nabla^\perp \p$ where $\psi \in H^1_0(\widetilde \OM)$, and $\partial_{\tau} \p\equiv 0$ on $\partial \OM$.
\end{proof}

\begin{remark}\label{rem div}
  Let $(\Omega,b)$ be a smooth lake in the sense of \eqref{smooth lake}. If $\psi\in X$ and $\F$ is a simili  harmonic function, then 
 $bv:=\nabla^\perp(\p+\F)$ verifies
 \[ \div(b v)=0 \text{ in } \Omega, \quad bv\cdot \nu = 0 \text{ on } \partial\Omega, \quad \text{(in the sense of \eqref{imperm})}.\]
 Indeed,  Proposition \ref{prop basis} states that there exists $c_k$ such that $\F\equiv \sum_{k=1}^N c_k \f^k$, so using \eqref{defi phi}, we can decompose $bv$ as $bv = \nabla^\perp(\tilde \p + \sum_{k=1}^N c_k \chi^k)$ with $\tilde \p\in X$. Then Lemma \ref{lem div} can be applied.
\end{remark}

With the regularity considered in Definition \ref{defi simili}, it is not clear that $\int_{\partial \Cc^k} \frac{\nabla^\perp \F}b\cdot \hat \tau \, ds$ is well defined. Using $\chi^k$ defined as in \eqref{chi}, we introduce the generalized circulation of a vector field $v$ around $\Cc^k$ by
\begin{equation}\label{circulation}
\g^k(v):= \int_{\OM} \div\Bigl[ \chi^k v^\perp\Bigl]\, dx = -\int_{\OM} \curl \Bigl[ \chi^k v\Bigl]\, dx =-\int_{\OM}\left(\na^\perp \chi^k\cdot v + \chi^k \curl v\right)\, dx .
\end{equation}

\begin{lemma}\label{1to1} Let $(\Omega,b)$ be a smooth lake in the sense of \eqref{smooth lake}.
If $\p$ is a simili harmonic function such that the generalized circulation of the vector field $\frac{\na^\perp \p}{b}$ around each $\mathcal{C}^k$ is equal to zero for all $k$, then $\p$ must be identically zero.
\end{lemma}

\begin{proof}
Set
\[v:=\frac{1}{b}\nabla^\perp\psi.\]
We begin the proof with the following claim: 
there exists a function $f$ such that
\begin{equation}\label{Grad}
v=\nabla f.
\end{equation}
We observe that 
\begin{equation}\label{local ellip}
\curl (v)=\div (\frac{1}{b}\nabla \psi)=0.
\end{equation}
From Remark \ref{rem div}, we get that
 \[ \div(b v)=0 \text{ in } \Omega, \quad bv\cdot \nu = 0 \text{ on } \partial\Omega, \quad \text{(in the sense of \eqref{imperm})}.\]
As $b$ is regular, by local elliptic regularity we deduce from \eqref{local ellip} that $v$ is a continuous function in $\OM$. Thus we can define the classical circulation $\oint v\cdot \tau ds$ along any closed path and we infer by the curl free property that this circulation does not depend on the homotopy class of the path. Next, choose $c_k$ a closed curve supported in the region where $\chi^k=0$ so that $c_k$ is homotopic to $\partial\mathcal{C}^k$ (see \eqref{chi} for the definition of $\chi^k$). Let $c^\prime_k$ be another homotopic path supported in $\{x\in \OM,\ \chi^k(x)=1\}$. We let $A^k$ be the region bounded by $c_k$ and $c^\prime_k$. Using \eqref{circulation}, we then compute that
\begin{equation*}
\begin{split}
\int_{c_k}v\cdot\tau ds&=\int_{c^\prime_k}v\cdot\tau ds=\left(\int_{c^\prime_k}-\int_{c_k}\right)(\chi^kv)\cdot\tau ds=-\int_{A^k}\curl (\chi^kv)\, dx=-\int_{\Omega}\curl (\chi^kv)\, dx=0,
\end{split}
\end{equation*}
where we have used the fact that
\begin{equation*}
\curl (\chi^kv)=\chi^k\curl (v)-v^\perp\cdot\nabla\chi^k
\end{equation*}
vanishes outside of $A^k$.

Since any smooth loop can be decomposed as a concatenation of a finite number of loops which are either homotopic to the trivial loop or homotopic to one of the $c_k$'s, we have that the circulation of $v$ along any closed curve vanishes. Therefore, fixing $P$ an arbitrary point in $\OM$ and letting
\begin{equation*}
f(Q)=\int_{\gamma_{PQ}}v\cdot\tau ds
\end{equation*}
where $\gamma_{PQ}$ is any (smooth) path from $P$ to $Q$, we obtain \eqref{Grad}.

As $\psi$ is a simili harmonic function, we have that  $\frac{\nabla^\perp\psi}{\sqrt b}=\sqrt b v=\sqrt b\nabla f$ belongs to $L^2(\OM)$, hence,
\begin{equation}\label{bv2}
\int_{\Omega}b\vert v\vert^2\, dx=\int_{\Omega}\frac{\nabla^\perp\psi}{\sqrt b}\cdot \sqrt b\nabla f \, dx.
\end{equation}
Moreover, by Proposition \ref{prop basis} and \eqref{defi phi} we can decompose $\psi$ as
\[
\psi = \sum_{k=1}^N c_{k} \varphi^k= \psi^0 +  \sum_{k=1}^N c_{k} \chi^k
\]
where $\psi^0\in X$. By density of $C^\infty_{c}(\Omega)$ in $X$ (see Lemma \ref{lem density}) we have
\begin{equation*}
\int_{\Omega}\frac{\nabla^\perp\psi^0}{\sqrt b}\cdot \sqrt b\nabla f \, dx=\lim_{n\to \infty} \int_{\Omega}\frac{\nabla^\perp\psi_{n}}{\sqrt b}\cdot \sqrt b\nabla f \, dx=\lim_{n\to \infty}\int_{\Omega}{\nabla^\perp\psi_{n}}\cdot \nabla f \, dx=0
\end{equation*}
where we have integrated by parts and used that $\psi_{n}\in C^\infty_{c}(\Omega)$. Moreover, as $\chi^k$ is smooth and $\nabla \chi^k$ vanishes close to the boundary, we also have by an integration by parts:
\[
\int_{\Omega}\frac{\nabla^\perp\chi^k}{\sqrt b}\cdot \sqrt b\nabla f \, dx=\int_{\Omega}{\nabla^\perp\chi^k}\cdot \nabla f \, dx=0
\]
for all $k$.

Putting together these two relations, \eqref{bv2} implies that $v$ is equal to zero, from which we conclude that $\psi$ is constant in $\Omega$. Since $\psi\in H^1_0(\widetilde\OM)$, $\psi$ vanishes in $\Omega$ as claimed.
\end{proof}

\begin{proposition}\label{Psi}
Let $(\Omega,b)$ be a smooth lake in the sense of \eqref{smooth lake}.
There exists a basis $\{\psi^k\}_{k=1}^N$ of $ \mathcal{SH}$ which satisfies
\begin{equation*}
\g^i\Bigl(\frac{\nabla^\perp \psi^k}{b}\Bigl)=\delta_{ik}, \qquad \forall ~ i,k.
\end{equation*}

\end{proposition}

\begin{proof} Consider the linear mapping
\begin{equation*}
\Psi: \mathcal{SH}\to \mathbb{R}^N, \quad \Psi(g) = (\gamma_1,\dots,\gamma_N), \quad \gamma_i:=\g^i\Bigl(\frac{\nabla^\perp g}{b}\Bigl) .
\end{equation*}
Lemma \ref{1to1} states that $\Psi$ is one-to-one and Proposition \ref{prop basis} implies that ${\rm dim}\  \mathcal{SH}=N$. Consequently, $\Psi$ is onto and we can define $\psi^i=\Psi^{-1}(e_i)$ where $e_i$ is the $i$-th vector in the canonical basis of $\mathbb{R}^N$. \end{proof}

\begin{proposition}[Decomposition]\label{decomposition}
Let $(\Omega,b)$ be a smooth lake in the sense of \eqref{smooth lake}.
 Let $\om \in L^\infty(\OM)$ and $\gamma_{0} = (\gamma_0^1,\cdots,\gamma_0^N) \in \R^N$. Then there exists a unique vector field $v$ such that $\sqrt{b}v\in L^2(\OM)$,
\begin{equation}\label{decom1}
  \div(b v)=0 \text{ in } \Omega, \quad bv\cdot \nu = 0 \text{ on } \partial\Omega, \quad \text{(in the sense of \eqref{imperm})}
\end{equation}
and
\begin{equation}\label{decom2}
 \curl (v) = b\omega \quad\text{ in }\quad \mathcal{D}'(\OM),  \qquad \gamma^i(v)=\gamma^i_{0}.
\end{equation}
Moreover, we have the following Biot-Savart formula:
\begin{equation}\label{BiotSavart}
v=b^{-1}\nabla^\perp \psi^0+\sum_{i=1}^N\alpha^i b^{-1}\nabla^\perp\psi^i,
\end{equation}
where $\psi^i\in\mathcal{SH}$ is the function defined as in Proposition \ref{Psi} above,
$\alpha^i=\gamma^i_{0}+\int_{\OM}b\omega\varphi^i \, dx, $ $\varphi^i$ defined as in Proposition \ref{prop basis},
and $\psi^0\in X$ the unique solution (see Proposition \ref{prop elliptic}) of the problem
\begin{equation*}
\div (\frac{1}{b}\nabla \psi^0)=b\omega \quad \text{ in } \mathcal{D}'(\OM),\quad\psi^0\in X .
\end{equation*}

\end{proposition}

\begin{proof} We begin by showing that the vector field defined as
\[
u:=b^{-1}\nabla^\perp \psi^0+\sum_{i=1}^N \Bigl(\gamma^i_{0}+\int_{\OM}b\omega\varphi^i \, dx \Bigl) b^{-1}\nabla^\perp\psi^i
\]
verifies \eqref{decom1}-\eqref{decom2}. By definition, $\sqrt{b}u\in L^2(\OM)$. The curl condition in \eqref{decom2} is obvious from the definitions of $\psi^0$ and the simili harmonic functions. Condition \eqref{decom1} comes from Remark \ref{rem div}. The hardest part is to compute the circulation  of  $b^{-1} \na^\perp \psi^0$. By the definition \eqref{defi phi} of $\f^i$, we use that $\chi^i = \f^i - \tilde \f^i$ with $\tilde \f^i \in X$, to get:
\begin{equation}\label{circu of K}
\g^i(b^{-1} \na^\perp \psi^0) = -\int_{\OM}\left(\na^\perp \f^i  \cdot b^{-1} \na^\perp \psi^0 + \f^i b \om \right)\, dx +\int_{\OM} \div\Bigl[ \tilde \f^i b^{-1} \na \psi^0 \Bigl]\, dx.
\end{equation}
Now, for any $\F \in C^\infty_c(\OM)$, we note that
\[\int_{\OM} \div\Bigl[  \F b^{-1} \na \psi^0 \Bigl]\, dx =0,\]
and as $\tilde \f^i$ belongs to $X$ and $C^\infty_c$ is dense in $X$, we deduce from the fact that both $b^{-1/2} \na \psi^0$ and $b\om$ belong to $L^2(\OM)$ that
\begin{equation}\label{cancellation}
\int_{\OM} \div\Bigl[ \tilde \f^i b^{-1} \na  \psi^0 \Bigl]\, dx =\int_{\OM}\Bigl[ b^{-1/2}(\na \tilde \f^i - \na \F)\cdot b^{-1/2} \na  \psi^0 - (\tilde \f^i-\F) b\om  \Bigl]\, dx=0.
\end{equation}
Moreover, that $\psi^0\in X$ allows us to integrate by parts the first term on the right hand side of \eqref{circu of K}, giving
\[ -\int_{\OM} \na^\perp \f^i  \cdot  b^{-1}\na^\perp \psi^0  \, dx=-\int_{\OM} b^{-1}\na \f^i  \cdot  \na \psi^0 \, dx = \int_{\OM}\div \left(b^{-1} \na \f^i \right)  \psi^0 \, dx=0.\]
The last identity was due to the fact that $\f^i$ is a simili harmonic function. Therefore, putting these last two equalities together with \eqref{circu of K}  gives  that $\gamma^i(u)=-\int_{\OM}b\omega\varphi^i \, dx +  \Bigl(\gamma^i_{0}+\int_{\OM}b\omega\varphi^i \, dx \Bigl) =\gamma^i_{0}$,  which shows that $u$ verifies \eqref{decom1}-\eqref{decom2}.

To prove the uniqueness of $u$, let $v$ be another vector field such that $\sqrt{b}v\in L^2(\OM)$ and $v$ satisfies \eqref{decom1}-\eqref{decom2}.

By Lemma \ref{lem div},  there exists $\psi \in  H^1_{0}(\widetilde\Omega)$ such that
\[ bv = \nabla^\perp \psi \text{ in } \mathcal{D}'(\overline{\OM}), \quad \pd_{\t} \p =0 \text{ on } \pd \OM.\]
So 
\[\Psi:= \psi-\Bigl(  \psi^0+\sum_{i=1}^N \Bigl(\gamma^i_{0}+\int_{\OM}b\omega\varphi^i \, dx \Bigl) \psi^i\Bigl)\]
is a simili harmonic function such that the circulation of $\frac{\na^\perp \P}{b}$ around each $\mathcal{C}^k$  is equal to zero. Lemma \ref{1to1} gives that $v=u$, which ends the proof.
\end{proof}

\subsection{Existence of a global weak solution}\label{sect 2.2}

In this subsection we prove the existence of a global weak interior solution for the vorticity formulation:

\begin{lemma}\label{WeakIntSol} Let $(\Omega,b)$ be a smooth lake in the sense of \eqref{smooth lake}.
Let $\omega^0\in L^\infty(\OM)$ and $\{\gamma_0^i\}_{1\le i\le N}$ fixed numbers and define $v^0$ by \eqref{BiotSavart}. Then there exists a global weak interior solution to the lake equations on $(\Omega,b)$ in the vorticity formulation (see Definition \ref{defi-weak-sol}). Moreover, the circulations of this solution are conserved.
\end{lemma}

The original idea comes from Yudovich: we introduce an artificial viscosity $\varepsilon \div(b \nabla \omega_\varepsilon)$ in the vorticity equations, assuming the Dirichlet condition for the vorticity at the boundary. This viscosity is artificial, because of the boundary condition: in the (physically relevant) Navier-Stokes equations, the Dirichlet condition is given on $v$, which does not imply the Dirichlet condition on the vorticity. Although the inviscid problem in Navier-Stokes equations is a hard issue, the limit problem $\varepsilon\to 0$ with the artificial viscosity is possible to achieve. Actually, to use directly a result from \cite{LOT1}, we consider $b_{\varepsilon} :=b+\varepsilon\geq \varepsilon>0$, and we approximate $\om^0$ by $\om^0_\varepsilon \in C^\infty_{c}$ such that $\|\om^0_\varepsilon\|_{L^\infty}\leq 2\|\om^0\|_{L^\infty}$. As $b_{\varepsilon}$ is strictly positive, standard arguments for Navier-Stokes equations \cite{LOT1} gives the existence and uniqueness of a global solution 
$$\omega_\varepsilon \in C([0,\infty);H^1_{0}(\Omega))\cap L^2_{\loc} ([0,\infty);H^2(\Omega))$$
of the problem (in the sense of distribution)
\begin{equation}\label{AppEq}
\left\lbrace \begin{aligned}
&\pd_t (b_\varepsilon  \omega_\varepsilon  ) + b_\varepsilon  v_\varepsilon  \cdot \nabla \omega_\varepsilon -\varepsilon \div(b_\varepsilon \nabla \omega_\varepsilon) =0 & \text{ for }(t,x)\in \R_+ \times \OM , \\
&v_\varepsilon=\frac1{b_\varepsilon} \nabla^\perp \psi^0_\varepsilon[\omega_\varepsilon]+\sum_{i=1}^N\frac{\gamma^i_{0}+\int_{\OM}b_\varepsilon\omega_\varepsilon\varphi_{\varepsilon}^i \, dx}{b_\varepsilon}\nabla^\perp\psi_{\varepsilon}^i,
 & \text{ for }(t,x)\in \R_+ \times \OM  ,\\
&\omega_\varepsilon  =0 & \text{ for }(t,x)\in \R_+ \times \pd \OM , \\
&\omega_\varepsilon (0,x) = \omega_\varepsilon^0 (x) & \text{ for }x\in  \OM ,
\end{aligned}
\right.
\end{equation}
where $\varepsilon>0$ is arbitrary and $\gamma^i_{0}$ are given independently of $\varepsilon$ and $t$.
The above system is exactly the problem studied in \cite{LOT1}. Indeed the authors work in non-simply connected domains, and Lemma 5 therein is similar to our decomposition (Proposition \ref{decomposition}). In their case, the tangential part $v\cdot \tau$ is clearly defined (as $b_\varepsilon>0$) so their definition of the circulation as an integral along $\partial \Cc^k$ is the same as our weak circulation. In this work, the test functions are compactly supported in $\Omega$: $\varphi \in C^\infty_c([0,\infty) \times \Omega)$. Indeed, for Navier-Stokes equations, the general framework is of $H^{-1}$ to $H^1_{0}(\Omega)$ and test functions in $C^\infty_c([0,\infty) \times \Omega)$ are sufficient because the Dirichlet boundary condition is already encoded by the fact that the velocity (here the vorticity) belongs to $H^1_{0}$. Moreover, we have the ``energy relation'':
\begin{equation}\label{energy}
\|\sqrt{b_{\varepsilon}}\om_\varepsilon(t,\cdot)\|_{L^2(\Omega)}^2 + \e \int_0^t \|\sqrt{b_{\varepsilon}} \nabla \om_\varepsilon(s,\cdot)\|_{L^2(\Omega)}^2  \, ds \leq \| \sqrt{b_{\varepsilon}} \om^0_{\varepsilon}\|_{L^2(\Omega)}^2, \quad \forall t\geq 0.
\end{equation}

Next, the idea is to pass to the limit $\varepsilon\to 0$. Let us perform this limit as follows:
\begin{itemize}
\item by integration by parts and Poincar\'e inequality on $\widetilde \Omega$ we have that (thanks to the tangency condition of $v_{\varepsilon}$):
\[\begin{aligned}
\| \sqrt{b_\varepsilon} v_\varepsilon \|^2_{L^2(\Omega)} &\leq \| b_\varepsilon \omega_\varepsilon \|_{L^2(\Omega)} \|\psi_\varepsilon \|_{L^2(\Omega)} \leq  C_{2}\sqrt{M+\varepsilon} \|  \sqrt{b_\varepsilon} \omega^0_{\varepsilon} \|_{L^2(\Omega)} \| \nabla \psi_\varepsilon \|_{L^2(\Omega)} 
\\& \leq 2C_{2} |\Omega|^\frac{1}{2} (M+1)^{3/2}  \| \omega^0 \|_{L^\infty(\Omega)}  \|  \sqrt{b_\varepsilon} v_\varepsilon \|_{L^2(\Omega)},
\end{aligned}\]
where $\psi_\varepsilon$ is the stream function vanishing on $\partial\widetilde\Omega$ associated to $b_{\varepsilon}v_\varepsilon$:
\[\psi_\varepsilon:=  \psi^0_\varepsilon[\omega_\varepsilon]+\sum_{i=1}^N (\gamma^i_{0}+\int_{\OM}b_\varepsilon\omega_\varepsilon\varphi_{\varepsilon}^i \, dx) \psi_{\varepsilon}^i. \]
Hence $\sqrt{b_\varepsilon} v_\varepsilon$ is uniformly bounded in $L^\infty(\R_{+};L^2(\Omega))$, uniformly in $\e$:
\begin{equation}\label{B}
\Vert\sqrt{b_\varepsilon}v_\varepsilon(t)\Vert_{L^2}\lesssim 1.
\end{equation}
\item for $\varepsilon$ fixed, we easily observe that $\partial_{t} \omega_\varepsilon \in L^2_{\loc}(\R_{+};H^{-1}(\Omega))$ and also that $\omega_\varepsilon \in C(\R_{+};L^2(\Omega))\cap L^2_{\loc}(\R_+;H^2(\Omega))$. Hence one can multiply the vorticity equation by some power of $\omega_{\varepsilon}$ to get for all time:
\[
\| (b_\varepsilon)^{\frac1p} \omega_\varepsilon (t,\cdot)\|_{L^p} \leq \| (b_\varepsilon)^{\frac1p} \omega^0_\varepsilon \|_{L^p}  \leq (M+1)^\frac{1}{p} \|  \omega^0_\varepsilon \|_{L^p} \leq 2[(M+1)(|\Omega|+1)]^\frac{1}{p} \| \omega^0 \|_{L^\infty} \ \forall p\in [1,\infty).
\]
As the constant at the right hand side is uniform in $p$, we infer that
\begin{equation}\label{om eps inf}
 \|  \omega_\varepsilon(t,\cdot) \|_{L^\infty} \leq  2\| \omega^0 \|_{L^\infty}.
\end{equation}
\end{itemize}

Therefore,  Banach-Alaoglu theorem implies that $ \omega_\varepsilon$ converges weak-$*$ to $\om$ in $L^\infty(\R_+\times\OM)$, and $\sqrt{b_{\varepsilon}} v_\varepsilon$ converges weak-$*$ to $\sqrt{b} v$ in $L^\infty(\R_+;L^2(\OM))$. This weak convergence is sufficient to get (i), (ii) and (iii) in Definition \ref{defi-weak-sol}. Moreover, by construction (see \eqref{AppEq}), $\gamma^{i}(v_{\varepsilon}(t))=\gamma^i_{0}$ for all $t\in \R_+$, $i=1\dots N$. Hence, the weak limit is also sufficient to pass to the limit in the circulation definition \eqref{circulation} which implies that the circulations of $v$ are conserved.

To get (iv), we will pass to the limit in equation \eqref{AppEq}, but we need a strong convergence of the velocity.
It would be tempting to use a variant of the Div-Curl lemma on $F_{\e}\cdot G_{\e}$ with $F_{\e}:=b_\e v_\e$ (which is divergence free) and  $G_\e:=v_\e$ (we could prove that $\curl G_{\e}=\om_{\e}$ is precompact in $C([0,T]; H^{-1}_{\loc} (\R^2;\R))$). However, a subtle problem appears when we try to verify the precompactness of $F_\varepsilon$ and $G_{\e}$ in $C([0,T]; H^{-1}_{\loc} (\R^2;\R^2))$ (which is necessary to apply the Div-Curl lemma): because of the absence of boundary conditions on $\mathcal{O}\Subset \OM$,
the mapping $Id:L^2(\mathcal{O})\to \mathcal{V}^\prime(\mathcal{O})$, where $\mathcal{V}^\prime(\mathcal{O})$ is the dual of $\mathcal{V}(\mathcal{O}):=\{v\in H^1_0(\mathcal{O}),\,\,\hbox{div}(v)=0\}$ is not an embedding (indeed, it maps gradients of functions to $0$). 
This prevents us from getting suitable compactness property in $C([0,T]; H^{-1} (\mathcal{O}))$ and forces us to only seek strong convergence on some part of the velocity and to use a hidden cancellation property of the equations. We now turn to the details.
 
Without loss of generality, we may restrict ourselves to $\mathcal{O}$ a smooth simply connected open subset of $\Omega$ such that $\overline{\mathcal{O}}\subset\Omega$. We introduce the Leray projector $\mathbb{P}_{\mathcal{O}}$ from $L^2(\mathcal{O})$ to $\mathcal{H}(\mathcal{O})$ (see \eqref{impermbis} for the definition), i.e. $\mathbb{P}_{\mathcal{O}}$ is the unique operator such that 
\begin{equation}\label{LerayProj}
v_\e = \mathbb{P}_{\mathcal{O}} v_\e + \nabla q_\e,\quad \hbox{div}(\mathbb{P}_{\mathcal{O}}v_\e)=0, \quad (\mathbb{P}_{\mathcal{O}}v_\e)\cdot  \nu\vert_{\partial \mathcal{O}} =0.
\end{equation}
All the details about the Leray projector can be found e.g. in \cite{Galdi}. In particular, it is known that such a projector is orthogonal in $L^2$, hence by (H2) we have that
\begin{equation*}
\|  \mathbb{P}_{\mathcal{O}} v_\e \|_{L^2(\mathcal{O})}^2+ \|  \nabla q_\e \|_{L^2(\mathcal{O})}^2 \leq  \|  v_\e \|_{L^2(\mathcal{O})}^2\leq \theta_{\mathcal{O}}^{-1}\| \sqrt{b_{\e}} v_\e \|_{L^2(\OM)}^2
\end{equation*}
 which implies by \eqref{B} that $\nabla q_\e$ and $\mathbb{P}_{\mathcal{O}} v_\e$ converge  weak-$\ast$  in $L^\infty([0,T];L^2(\mathcal{O}))$ to $\nabla q$ and $\mathbb{P}_{\mathcal{O}} v$, with $v=\mathbb{P}_{\mathcal{O}} v +\nabla q$.
Besides, since $\hbox{curl}(\mathbb{P}_{\mathcal{O}}v_\e)=\hbox{curl}(v_\varepsilon)=b_\varepsilon\omega_\varepsilon$ is uniformly bounded in $L^\infty$ and using \eqref{LerayProj}, we see that $\{\mathbb{P}_{\mathcal{O}} v_\varepsilon(t)\}$ always remain inside a compact set of $L^2(\mathcal{O})$ (indeed, by the standard Calder\'on-Zygmund theorem on $\mathcal{O}$, $\sum_j\Vert \partial_j\mathbb{P}_{\mathcal{O}} v_\e(t)\Vert_{L^2}\lesssim \Vert \hbox{curl}(v_\varepsilon(t))\Vert_{L^2}$).

 As \eqref{AppEq} is verified for test functions $\varphi \in C^\infty_c( \mathcal{O})$ with $\mathcal{O}$ a simply connected smooth domain, like in Proposition \ref{prop A2} in Appendix \ref{sect equiv defi} we infer that we have a velocity type equation:
\begin{equation*}
 -\int_0^\infty\int_{\OM}\Phi \cdot \partial_{t} v_\e \, dxdt +\int_0^\infty \int_{\OM}\om_{\e}b_\e v_\e \cdot \Phi^\perp \, dxdt
 -\e \int_0^\infty \int_{\OM} b_\e \nabla \om_\e \cdot \Phi^\perp \, dxdt =0
\end{equation*}
for all divergence-free $\Phi \in C^\infty_c( \mathcal{O})$. For such a test function, using \eqref{energy}, \eqref{B} and \eqref{om eps inf}, we obtain that
\begin{equation*}
\begin{split}
\left\vert\langle \mathbb{P}_{\mathcal{O}}v_\varepsilon(t_2),\Phi\rangle-\langle \mathbb{P}_{\mathcal{O}}v_\varepsilon(t_1),\Phi\rangle\right\vert&=\left\vert\langle v_\varepsilon(t_2),\Phi\rangle-\langle v_\varepsilon(t_1),\Phi\rangle\right\vert\\
&=\left\vert\int_{t_1}^{t_2}\int_\OM b_\varepsilon\omega_\e v_\e\cdot \Phi^\perp  \, dxdt-\varepsilon\int_{t_1}^{t_2}\int_\OM b_\varepsilon\Phi^\perp\cdot\nabla\omega_\varepsilon \, dxdt\right\vert\\
&\lesssim \Vert\Phi\Vert_{L^2}\big[\sqrt{M+1}\Vert \sqrt{b_{\varepsilon}} v_\varepsilon\Vert_{L^\infty_tL^2}\Vert\omega_\varepsilon\Vert_{L^\infty_{t,x}}\vert t_1-t_2\vert\\
&\quad+\sqrt{\varepsilon(M+1)}\Vert\sqrt{\varepsilon b_\varepsilon}\nabla\omega_\varepsilon\Vert_{L^2_{t,x}}\vert t_1-t_2\vert^\frac{1}{2}\big]\\
&\lesssim  \Vert\Phi\Vert_{L^2}C(\omega^0)\left[\vert t_1-t_2\vert+\sqrt{\varepsilon}\vert t_1-t_2\vert^\frac{1}{2}\right].
\end{split}
\end{equation*}
By density, we note that the above estimates is true for $\F\in \mathcal{H}(\mathcal{O})$. Therefore, for any $\F \in L^2(\mathcal{O})$, we write that
\[
\langle \mathbb{P}_{\mathcal{O}}v_\varepsilon(t),\Phi\rangle = \langle \mathbb{P}_{\mathcal{O}}v_\varepsilon(t),\mathbb{P}_{\mathcal{O}}  \Phi + \nabla q_{\Phi} \rangle= \langle \mathbb{P}_{\mathcal{O}}v_\varepsilon(t),\mathbb{P}_{\mathcal{O}}  \Phi \rangle
\]
because $ \mathbb{P}_{\mathcal{O}}v_\varepsilon(t)\cdot  \nu \vert_{\partial \mathcal{O}}=0$ and $\div \mathbb{P}_{\mathcal{O}}v_\varepsilon(t)=0$. Hence, the above estimate can be used to get for any $\F \in L^2(\mathcal{O})$
\begin{equation*}
\begin{split}
\left\vert\langle \mathbb{P}_{\mathcal{O}}v_\varepsilon(t_2),\Phi\rangle-\langle \mathbb{P}_{\mathcal{O}}v_\varepsilon(t_1),\Phi\rangle\right\vert &\lesssim  \Vert\mathbb{P}_{\mathcal{O}} \Phi\Vert_{L^2}C(\omega^0)\left[\vert t_1-t_2\vert+\sqrt{\varepsilon}\vert t_1-t_2\vert^\frac{1}{2}\right]\\
&\lesssim  \Vert \Phi\Vert_{L^2}C(\omega^0)\left[\vert t_1-t_2\vert+\sqrt{\varepsilon}\vert t_1-t_2\vert^\frac{1}{2}\right]
\end{split}
\end{equation*}
which implies that the family $\{\mathbb{P}_{\mathcal{O}}v_\varepsilon\}$ is equicontinuous in $L^2(\mathcal{O})$. Since we have seen that it takes values in a compact set, Arzela-Ascoli gives us the precompactness of $\{\mathbb{P}_{\mathcal{O}}v_\e\}$ in $C([0,T];L^2(\mathcal O))$.

Finally, we can now pass to the limit in \eqref{AppEq}. We recall that for any $\varphi \in C^\infty_c([0,T)\times\mathcal{O})$, the first equation in \eqref{AppEq} reads 
\begin{equation}\label{AppEq-re}
\begin{split}
0=&\int_{0}^\infty \int_{\OM} \Big [ \varphi_{t}b_\varepsilon  \omega_\varepsilon +  b_\varepsilon \omega_\varepsilon v_\varepsilon    \cdot \nabla\varphi  - \varepsilon   b_\varepsilon \nabla \omega_\varepsilon \cdot\nabla \varphi \Big] \, dxdt + \int_{\OM} \varphi(0,x)b_\varepsilon  \omega_\varepsilon(x) \, dx.\\
\end{split}
\end{equation}
Clearly, thanks to \eqref{energy}, we can pass to the limit as $\varepsilon \to 0$ in all the (linear) terms except the nonlinear term: $b_\varepsilon \omega_\varepsilon v_\varepsilon    \cdot \nabla\varphi $. For the remaining term, using the relation \eqref{algebra}, we get
\begin{equation*}
\begin{split}
\int_{0}^\infty \int_{\OM}  b_\varepsilon \omega_\varepsilon v_\varepsilon    \cdot \nabla\varphi  \, dxdt  & =\int_{0}^\infty \int_{\OM}  (\curl v_\varepsilon) v_\varepsilon^\perp    \cdot \nabla^\perp\varphi  \,dxdt \\
& = \int_{0}^\infty \int_{\OM} \Big[  \div(b_{\varepsilon}v_{\varepsilon}\otimes v_{\varepsilon})\cdot \frac{\nabla^\perp \varphi}{b_{\varepsilon}}  -\frac12 \nabla |v_{\varepsilon}|^2   \cdot \nabla^\perp\varphi \Big ]  \, dxdt \\
&=  \int_{0}^\infty \int_{\OM} \div(b_{\varepsilon}v_{\varepsilon}\otimes v_{\varepsilon})\cdot \frac{\nabla^\perp \varphi}{b_{\varepsilon}} \, dxdt .
\end{split}
\end{equation*}
In addition, we can write
\[
b_\e v_\e \otimes v_\e  = b_\e \mathbb{P}_{\mathcal O} v_\e \otimes v_\e + b_\e \nabla q_\e \otimes \mathbb{P}_{\mathcal O} v_\e + b_\e \nabla q_\e \otimes \nabla q_\e, \]
in which $\mathbb{P}_{\mathcal{O}}$ is the Leray projector defined as above. The integration involving the first two terms on the right hand side converges to its limit by taking integration by parts and using a weak-strong convergence argument. For the last term, we further compute:
\[\div[b_\e \nabla q_\e \otimes \nabla q_\e] =\frac{b_\varepsilon}2 \nabla(|\nabla q_{\varepsilon}|^2)+(\nabla b_{\varepsilon}\cdot \nabla q_{\varepsilon}+b_{\varepsilon}\Delta q_{\varepsilon})\nabla q_{\varepsilon}.
\]
Here we note from \eqref{LerayProj} that $\div v_{\varepsilon}=\Delta q_{\varepsilon}$, and as $\div b_{\varepsilon}v_{\varepsilon}=0$, we get that $\Delta q_{\varepsilon}=-\frac{\nabla b_{\varepsilon}}{b_{\varepsilon}}\cdot v_{\varepsilon}$. Hence, we have 
\[
\div[b_\e \nabla q_\e \otimes \nabla q_\e] = \frac{b_\e}2 \nabla ( |\nabla q_\e |^2)-(\nabla b_\e \cdot  \mathbb{P}_{\mathcal O} v_\e) \nabla q_\e.
\]
This yields 
$$\begin{aligned}
\int_{0}^\infty & \int_{\OM}   \div(b_{\varepsilon}v_{\varepsilon}\otimes v_{\varepsilon})\cdot \frac{\nabla^\perp \varphi}{b_{\varepsilon}} \, dxdt  = \int_{0}^\infty \int_{\OM}\Big[ \frac12\nabla ( |\nabla q_\e |^2)\cdot \nabla^\perp \varphi  + (\nabla b_\e \cdot  \mathbb{P}_{\mathcal O} v_\e) \nabla q_\e \cdot\frac{\nabla^\perp \varphi}{b_{\varepsilon}}  \Big] \, dxdt 
\end{aligned}$$
in which the first integral vanishes, whereas the second integral passes to the limit again by a weak-strong convergence argument. 

By putting these altogether into \eqref{AppEq-re} and using the same algebra as just performed, it follows in the limit that
\begin{equation}\label{limit-AppEq}
\int_{0}^\infty \int_{\OM} \varphi_{t}b  \omega \, dxdt + \int_{0}^\infty \int_{\OM} \nabla\varphi\cdot vb\omega \, dxdt 
+ \int_{\OM} \varphi(0,x)b  \omega^0(x) \, dx=0\end{equation}
for all $\varphi \in C^\infty_c([0,T)\times \mathcal{O})$. Recall that $\mathcal{O}$ was an arbitrary smooth simply connected domain in $ \OM$. This proves that the identity \eqref{limit-AppEq} holds for all $\varphi \in C^\infty_c([0,\infty)\times \OM)$.

To conclude, we have shown that $(v,\omega)$ is an interior weak solution of the lake equations in the vorticity formulation, which completes the proof of Lemma \ref{WeakIntSol}.


\subsection{Well-posedness of a global weak solution}\label{sec-uniqueness}

In this subsection, we use the Calder\'on-Zygmund inequality \eqref{CZ} of Bresch and M\'etivier \cite{BM} to upgrade our solution $(v,\omega)$ to a weak solution in the vorticity formulation, which is then equivalent to a weak solution in the velocity formulation. Using again \eqref{CZ}, we prove that weak solutions in the velocity formulation are unique, which ends the proof.

\noindent

\noindent{\it Gain of regularity for smooth lakes.} First, we recall the main result of Bresch and M\'etivier in \cite{BM}: if the lake is smooth with constant slopes, then we have a Calder\'on-Zygmund  type inequality. Namely:

\begin{proposition}[{\cite[Theorem 2.3]{BM}}]\label{prop CZ} Let $(\Omega,b)$ be a smooth lake in the sense of \eqref{smooth lake}. Let $f\in L^p(\Omega)$ for $p>4$ and $bv\in L^2(\Omega)$. If the triplet $(b,v,f)$ verifies the following elliptic problem
 \[ \div(b v)=0 \text{ in } \Omega, \quad bv\cdot \nu = 0 \text{ on } \partial\Omega, \quad \text{(in the sense of \eqref{imperm})}\]
and
\[ \curl v = f   \text{ in } \mathcal{D}'(\OM),\]
 then $v\in C^{1-\frac2p}(\overline{\Omega})$ and $\nabla v\in L^p(\Omega)$. Moreover, there exists a constant $C$ which depends only on $\Omega$ and $b$ so that  for any $p>4$ 
\begin{equation}\label{CZ}
 \| \nabla v \|_{L^p(\Omega)} \leq Cp\Bigl(\| f \|_{L^p(\Omega)}+ \| bv \|_{L^2(\Omega)}\Bigl).
\end{equation}
In addition,
\[
v \cdot \nu =0 \text{ on } \partial \OM.
\]
\end{proposition}
 This inequality is well known in the case of non-degenerating topography ($b\geq \theta_{0}>0$) and it was extended by Bresch and M\'etivier in the case of a depth which vanishes at the shore like $d(x)^a$ for $a>0$.
 The authors decompose the domain in two pieces: one which is far from the boundary where they use classical elliptic estimates, and one near the boundary. As for the latter piece, they flatten the boundary and are reduced to study a degenerate elliptic equation with coefficients vanishing at the boundary of a half-plane. 

This decomposition in several subdomains explains why we have the terms $\|bv\|_{L^2}$ in the right hand side part of the Calder\'on-Zygmund inequality \eqref{CZ}, coming from the support of the gradient of some cut-off functions. We remark also that we can easily have some islands with vanishing (where $a_{k}$ can be different from $a_{0}$) or non vanishing topography, which gives a lake where the Calder\'on-Zygmund inequality holds true. 

\medskip

By Lemma \ref{WeakIntSol} there exists $(v,\omega)$ verifying the elliptic problem ii)-iii) in Definition \ref{defi-weak-sol}. Then, Proposition \ref{prop CZ} states that $\nabla v$ belongs to $L^p$ for any $p>4$. This estimate is crucial to prove that $(v,\omega)$ is actually a global weak solution to the vorticity formulation (Proposition \ref{prop A4}), which is also a global weak solution to the velocity formulation (Proposition \ref{prop A3}), because the circulations are conserved. The Calder\'on-Zygmund inequality will be also the key for the uniqueness. 

By using the renormalized solutions in the sense of DiPerna-Lions, it follows that $\omega\in C([0,\infty);L^p(\Omega))$ and $v\in C([0,\infty);W^{1,p}(\Omega))$ for any $p>4$ (see the proof of Lemma \ref{lem-vorticity} for details about the renormalized theory).

\bigskip

\noindent{\it Uniqueness.} The uniqueness part now follows from the celebrated proof of Yudovich \cite{yudo}. Let $v_{1}$ and $v_{2}$ be two weak global solutions for the same initial $v^0$. We introduce $\tilde v:=v_{1}-v_{2}$. As $\tilde v$ belongs to $W^{1,p}$ for any $p\in (4,\infty)$, we get from the velocity formulation some estimates for $\partial_{t}\tilde v$. This allows us to replace the test function by $b\tilde v=\mathbb{P}_{\OM}(b\tilde{v})\in C^1([0,T];L^\frac{5}{4}(\Omega))$. As $\tilde v\in  C\left(\R_+, L^5(\RR)\right)$, we get for all $T\in\R^+$
\begin{equation*}
\| \sqrt{b} \tilde v(T)\|_{L^2(\Omega)}^2=2\int_0^T \langle \pd_t (b\tilde v), \tilde v \rangle_{L^\frac{5}{4}\times L^5}\,ds\leq 2\int_0^T \int_{\Omega} |\sqrt{b} \tilde v(s,x)| |\nabla v_{2}(s,x)|  |\sqrt{b} \tilde v(s,x)| \, dxds
\end{equation*} 
where we have used that $\div bv_{1}=\div b\tilde v=0$. Next, we use the Calder\'on-Zygmund inequality \eqref{CZ} on $\nabla v_{2}$ to infer by interpolation that 
\[ \|  \sqrt{b} \tilde v (T,\cdot)\|_{L^2}^2 \leq 2Cp\int_0^T \| \sqrt{b} \tilde v\|_{L^2}^{2-2/p}\,dt.\]
Together with a Gronwall-like argument, this implies
\begin{equation*}
\|  \sqrt{b} \tilde v (T,\cdot)\|_{L^2}^2  \leq (2CT)^p,\qquad \forall p\geq 2.
\end{equation*}
Letting $p$ tend to infinity, we conclude that $\| \sqrt{b} \tilde v (T,\cdot)\|_{L^2} = 0$ for all $T<1/(2C)$. Finally, we consider the maximal interval of $[0,\infty)$ on which $\| \sqrt{b} \tilde v (T,\cdot)\|_{L^2} \equiv 0$, which is closed by continuity of $\| \sqrt{b} \tilde v (T,\cdot)\|_{L^2}$. If it is not equal to the whole of $[0,\infty)$, we may repeat the above proof, which leads to a contradiction by maximality. Therefore uniqueness holds on $[0,\infty)$, and this concludes the proof of well-posedness. 
\bigskip

\noindent
{\it Constant circulation.}
If the domain is not simply connected, we have proved in the first subsection that the vorticity alone is not sufficient to determine the velocity uniquely, and that we need to fix the weak circulation to  derive the Biot-Savart law. In the following section, the main idea is to prove compactness in each terms in this Biot-Savart law. Therefore, it is crucial to establish the Kelvin's theorem in our case, namely the weak circulations are conserved. Fortunately, this is valid in a great generality following Proposition \ref{circ const} as follows.

\begin{proposition}\label{circ const} Let $(\Omega,b)$ be a lake satisfying (H1)-(H2) with $b\in W^{1,\infty}_{\loc}(\Omega)$. Let $v$ be a global interior weak solution of the velocity formulation and a global weak solution of the vorticity formulation. Then for each $k = 1,\cdots,N$, the generalized circulation $\gamma^k$ defined as in \eqref{circulation} is independent of $t$.
\end{proposition}

\begin{proof}

Let $l(t)\in C^\infty_c([0,\infty))$, note that since $\nabla^\perp\chi^k\equiv0$ in a neighborhood of the boundary, then $l(t)\nabla^\perp\chi^k(x)$ is a test function for which \eqref{eq-vel} is verified. As  $\chi^k$ is constant in each neighborhood of the boundary, $l(t)\chi^k(x)$ is a test function for which \eqref{eq-vort} holds. Then, we can compute
\begin{equation*}
\begin{split}
\int_{\mathbb{R}}\gamma^k(t)\frac{d}{dt}l(t)dt - \gamma^k(0)l(0)
&=-\int_{\mathbb{R}}\int_{\OM}\frac{d}{dt}\left[l(t)\nabla^\perp\chi^k\right]\cdot v\, dxdt-\int_{\mathbb{R}}\int_{\OM}\frac{d}{dt}\left[l(t)\chi^k\right]b\omega \, dxdt- \gamma^k(0)l(0)\\
&=\int_{\mathbb{R}}\int_{\OM}\left\{(bv\otimes v):\nabla\left[\frac{1}{b}\nabla^\perp\chi^k\right]+\nabla\chi^k\cdot v\,\hbox{curl}(v)\right\}l(t)\, dxdt\\
&\qquad+l(0)\int_{\OM}\left\{v^0\cdot\nabla^\perp\chi^k+\chi^k\hbox{curl}(v^0)\right\}\, dx- \gamma^k(0)l(0)\\
&=\int_{\mathbb{R}}\int_{\OM}\left\{(bv\otimes v):\nabla\left[\frac{1}{b}\nabla^\perp\chi^k\right]+\nabla\chi^k\cdot v\,\hbox{curl}(v)\right\}l(t)\, dxdt.
\end{split}
\end{equation*}
Using the fact that $\hbox{div}(bv)=0$ and $\nabla^\perp\chi^k\equiv0$ in a neighborhood of the boundary, we may integrate by parts and use \eqref{algebra} to have
\begin{equation*}
\begin{split}
\int_{\mathbb{R}}\gamma^k(t)\frac{d}{dt}l(t)dt - \gamma^k(0)l(0)
&=-\int_{\mathbb{R}}l(t)\int_{\OM}\nabla^\perp\chi^k\cdot\nabla\frac{\vert v\vert^2}{2}\, dxdt.
\end{split}
\end{equation*}
Now, we let $\widetilde{\chi}^k$ be a smooth function, compactly supported inside $\Omega$ and  such that $\widetilde{\chi}^k\nabla\chi^k=\nabla\chi^k$. Integrating by parts, we then find that
\begin{equation*}
\begin{split}
\int_{\OM}\nabla^\perp\chi^k\cdot\nabla\frac{\vert v\vert^2}{2}\, dx&=\int_{\OM}\nabla^\perp\chi^k\cdot\nabla\left[\widetilde{\chi}^k\frac{\vert v\vert^2}{2}\right]\, dx=-\int_{\OM}\hbox{div}\left\{\nabla^\perp\chi^k\right\}\widetilde{\chi}^k\frac{\vert v\vert^2}{2}\, dx=0.
\end{split}
\end{equation*}
This finishes the proof.
\end{proof}

\section{Proof of the convergence}\label{sect 3}

 \def\vn{{\bf n}} 
 
In this section, we shall prove our main result (Theorem \ref{theo-main}). Here, we recall our main assumption that $(\OM_n,b_n)$ converges to the lake  $(\OM,b)$ as $n \to \infty$ in the sense of Definition \ref{def-ConvLake}. Let us denote by $D$ a large open ball such that $D$ contains $\Omega$ and $\Omega_n$, and extend the bottom functions $b$ and $b_n$ to zero on the sets $D\setminus \Omega$ and $D\setminus \Omega_n$, respectively. 

\medskip

We prove the main theorem via several steps. First, from the velocity equation, it is relatively easy to obtain an a priori bound on $\sqrt {b_n} v_n$ in $L^\infty(\R_+; L^2)$ (here one needs uniform estimates on $\sqrt{b_n}v^0_n$ in $L^2(\Omega_n)$). Unfortunately, such a bound is too weak to give any reasonable information on the possible limiting velocity solution $v$. To obtain sufficient compactness, we derive estimates on the stream function $\psi_n$, defined by
\begin{equation}\label{def-streamfn}
v_n=\frac1{b_n} \na^\perp \p_n.
\end{equation}
The Biot-Savart law \eqref{BiotSavart} which is established in Proposition \ref{decomposition} gives
\begin{equation}\label{BS-law} 
\psi_n (t,x)= \psi_n^0(t,x) + \sum_{k=1}^N \alpha_n^k(t) \psi_n^k(x), 
\end{equation}
where for each $n$, $\psi_n^0$ solves $\div(b_n^{-1}\na \p^0_n) = b_n \om_n$ with the Dirichlet boundary condition on $\partial\Omega_n$, and the so-called simili harmonic functions $\psi_n^k$ solve $\div(b_n^{-1} \na \p^k_n) = 0$ and have their circulations equal to $\delta_{jk}$ around each island $\mathcal{C}_n^j$, $j = 1,\cdots,N$. The real numbers $\alpha_{n}^k(t)$ are given by 
$$\alpha^k_{n}(t)=\gamma^k_{n}+\int_{\OM}b_{n}(x)\omega_{n}(t,x)\varphi^k_n(x) \, dx$$ where $\gamma^k_n=\gamma^k(v_{n})$ is the circulation of $v_n$  around each $\mathcal{C}_n^k$ introduced as in \eqref{circulation}, which is constant in time (see Proposition \ref{circ const}).

\subsection{Vorticity estimates} We begin by deriving some basic estimates on the vorticity $\omega_n$. 
\begin{lemma}\label{lem-vorticity} For each $n$, the $L^p$ norm of $b_n^{1/p}\omega_n$ is conserved in time and uniformly bounded for all $p\ge 1$, that is,
\begin{equation*}
 \|b_n^{1/p} \omega_n(t)\|_{L^p(\Omega_n)} =  \|b_n^{1/p} \omega^0_n\|_{L^p(\Omega_n)}\lesssim\Vert\omega^0_n\Vert_{L^\infty}\lesssim 1, \qquad \forall t\ge 0. \end{equation*}
In addition, $\omega_n$ is bounded in $L^\infty_{x,t}$, uniformly in $n$.  
\end{lemma} 
\begin{proof} We recall that the vorticity $\omega_n$ solves \eqref{transport} in the distributional sense and belongs to $L^\infty (\R_+ \times \Omega_n)$ . Thanks to Proposition \ref{prop CZ} we deduce that the velocity is regular enough to apply the renormalized theory in the sense of DiPerna-Lions:  let $f:\R \rightarrow \R$ be a smooth function such that
\begin{equation*}
|f'(s)|\leq C(1+ |s|^p),\qquad \forall t\in \R,
\end{equation*}
for some $p\geq 0$, then $f(\om_n)$ is a solution of the transport equation  \eqref{transport} (in the sense of distribution) with initial datum $f(\om_0)$.

By smooth approximation of $s\mapsto |s|^p$ for $1\leq p<\infty$, the renormalized solutions yields 
\[ \frac{d}{dt}(b_n |\om_n|^p) = - b_n v_n \na |\om_n|^p=- \div (b_n v_n |\om_n|^p).\]
Integrating this identity over $\Omega_n$ and using the Stokes theorem, we get 
\[ \frac{d}{dt}\int_{ \Omega_n}b_n(x) |\om_n|^p(t,x) \, dx =  - \int_{\partial\Omega_n} b_n v_n \cdot {\nu} |\omega_n|^p d\sigma_n(x),\]
where the boundary term vanishes due to the boundary condition on the velocity (see \eqref {lake-eqs}). The lemma is proved for $1\leq p<\infty$.

The case $p=\infty$ is easily obtained by taking $f$ a function vanishing on the interval $[-2\| \omega_0\|_{L^\infty}, 2\| \omega_0\|_{L^\infty}]$ and strictly positive elsewhere. Indeed, it shows that the $L^\infty$ norm cannot increase, and by time reversibility that it is constant.
\end{proof}

Lemma \ref{lem-vorticity} in particular yields that the vorticity $\omega_n$ is bounded in $L^\infty(\Omega_n)$ and, after extending $\omega_n$ by $0$ in $D\setminus\Omega_n$, by the Banach-Alaoglu theorem, we can extract a subsequence such that
\begin{eqnarray*}
 b_n^{1/p} \om_n\rightharpoonup  b^{1/p} \om &\text{ weakly-$\ast$ in }& L^\infty(\R_+; L^p (D)) \\
\om_n \rightharpoonup \om  &\text{ weakly-$\ast$ in }& L^\infty (\R_+ \times D).
\end{eqnarray*}

\subsection{Simili harmonic functions: Dirichlet case} We now derive estimates for the simili harmonic solutions $\varphi_n^k$, $k=1,\cdots,N$. We recall that  $\varphi_n^k$ vanishes on the outer boundary $\partial \tilde \Omega_n$ and solves 
\begin{equation}\label{def-phikn} \left\{ \begin{aligned} \div \Big[ \frac 1{b_n} \nabla \varphi_n^k \Big] &=0, \qquad \mbox{in}\quad \Omega_n\\
\varphi_n^k & = \delta_{jk}, \qquad \mbox{on}\quad \partial \mathcal{C}^j_n,\qquad j=1,\cdots,N .\end{aligned}\right.\end{equation}
The existence and uniqueness of $\varphi_{n}^k$ was established in Proposition \ref{prop basis}. We obtain the following. 

\begin{lemma}\label{lem-similifns} 
The sequence $b_n^{-1/2} \nabla \varphi_n^k$ converges strongly to $b^{-1/2}\nabla\varphi^k$ in $L^2(D)$. In particular, $\varphi_n^k$ is uniformly bounded in $H^1(D)$ and $b_n^{-1/2} \nabla \varphi_n^k$ is uniformly bounded in $L^2(D)$. 
\end{lemma}

In this statement and in all the sequel, $b_n^{-1/2} \nabla \varphi_n^k$ is extended by zero on $D\setminus \Omega_{n}$.

\begin{proof}
We first prove the boundedness and obtain convergence as a result of the convergence of the norm. As before, it is convenient to write, as in Proposition \ref{prop basis},
\[ \varphi_n^k = \tilde \varphi_n^k +  \chi^k, \qquad k = 1,\cdots,N.\]
Here, $\tilde \varphi_n^k \in X_{b_n}$ and $\chi^k$ denote  the cut-off functions  in $C^\infty_c(\Omega)$ such that $\chi^k$ is supported in a neighborhood of $\Cc^k$ and is identically equal to one on a smaller neighborhood of $\Cc^k$. Since $\Cc_n^k$ converges to $\Cc^k$, without loss of generality we can further assume that the same assumptions hold for $\Cc_n^k$ uniformly in $n\ge 0$. We then obtain $\tilde \varphi_n^k$ by solving  
\begin{equation}\label{def-tphikn} \div \Bigl[ \frac1{b_n}\nabla  \tilde \varphi_n^k \Bigl] = - \div \Bigl[ \frac1{b_n} \nabla \chi^k \Bigl] ,\quad \text{ in }  \Omega_n, \qquad  \tilde \varphi_n^k =0 \quad \text{ on } \quad\partial \Omega_n.\end{equation}
Multiplying this equation by $\tilde \varphi_n^k$ and integrating the result over $\Omega_n$, we readily obtain an a priori estimate:
 $$ \int_{\Omega_n} \frac{1}{b_n} |\nabla \tilde \varphi_n^k |^2 \; dx  = -\int_{\Omega_n} \frac{1}{b_n}\nabla \tilde\varphi_n^k  \nabla \chi^k \; dx \le \frac 12 \int_{\Omega_n} \frac{1}{b_n} |\nabla \tilde \varphi_n^k |^2 \; dx  +\frac 12 \int_{\Omega_n} \frac{1}{b_n} |\nabla \chi^k|^2 \; dx .$$
Here, we have used the Dirichlet boundary condition on $\tilde \varphi_n^k$. Now, remark that $\nabla \chi^k$ vanishes identically on a neighborhood of the boundary $\partial \Omega_n$ and $b_n$ are bounded above and below away from $\partial \Omega_n$. The last integral on the right-hand side of the above estimate is therefore uniformly bounded in $n$. 

This proves the boundedness and the weak convergence of $b_n^{-1/2}\nabla \tilde \varphi_n^k$ in $L^2(D)$ (with zero extension on $D\setminus \Omega_n$). Therefore, $\nabla \tilde \varphi_n^k$ is uniformly bounded in $L^2(D)$. The $H^1$ boundedness of $\tilde \varphi_n^k$ follows at once by the standard Poincar\'e inequality. 

Consequently, solutions $\varphi_n^k$ to \eqref{def-phikn} converge weakly in $H^1(D)$ to $\varphi^k\in H^1_{0}(D)$ verifying (in the sense of distributions):
\begin{equation*}
\div \Big[ \frac 1{b} \nabla \varphi^k \Big] =0, \qquad \mbox{in}\quad \Omega.
\end{equation*}
Without assuming that $\Omega_{n}$ is an increasing sequence, the difficulty could be to prove that $\varphi^k$ satisfies the right boundary conditions. The tool to get the boundary conditions is the $\gamma$-convergence. Namely, as $\OM_{n}$ converges in the Hausdorff topology to $\OM$ and as $\R^2\setminus \OM_{n}$ has $N+1$ connected components, then Proposition \ref{prop9} states that $\OM_{n}$ $\gamma$-converges to $\OM$. Hence, we can apply Proposition \ref{prop10} to $\tilde \varphi_n^k$ and infer that $\tilde \varphi^k$ belongs to $H^1_{0}(\OM)$. Therefore, we have the right boundary conditions:
\begin{equation*}
\varphi^k  = \delta_{jk}, \qquad \mbox{on}\quad \partial \mathcal{C}^j,\qquad j=1,\cdots,N .
\end{equation*}
Now, from the boundedness of $b_n^{-1/2} \nabla \tilde \varphi_n^k$ in $L^2$, we obtain at once the integrability of $b^{-1/2} \nabla \tilde \varphi^k$. Thus, by definition, $\tilde \varphi^k\in X$. 

From the equation \eqref{def-tphikn}, the weak convergence obtained above, the fact that $\tilde\varphi^k\in H^1_{0}(\OM)$ and that $b_{n}^{-1}\to b^{-1}$ in $L^2(\supp \nabla \chi^k)$ (by Definition \ref{def-ConvLake}), we have that
$$ \int_{\Omega_n} \frac{1}{b_n} |\nabla \tilde \varphi_n^k |^2 \; dx  = -\int_{\Omega_n} \frac{1}{b_n}\nabla \tilde\varphi_n^k  \nabla \chi^k \; dx \to -\int_{\Omega} \frac{1}{b}\nabla \tilde\varphi^k  \nabla \chi^k \; dx =  \int_{\Omega} \frac{1}{b} |\nabla \tilde \varphi^k |^2 \; dx.$$
This proves the strong convergence as claimed.
\end{proof}

\subsection{Simili harmonic functions: constant circulation} We next derive the convergence for the simili harmonic solutions $\psi_n^k$. We recall that  $\psi_n^k$ vanishes on the outer boundary $\partial \tilde \Omega_n$ and solves 
\begin{equation}\label{def-psikn} 
\left\{ \begin{aligned} \div \Big[ \frac 1{b_n} \nabla \psi_n^k \Big] &=0, \qquad \mbox{in}\quad \Omega_n\\
\gamma_n^j\Big( \frac 1{b_n} \nabla^\perp \psi_n^k\Big) & = \delta_{jk}, \qquad j=1,\cdots,N .\end{aligned}\right.\end{equation}
where the circulation around $\mathcal{C}^k$ defined in \eqref{circulation} verifies
$$\gamma_n^j\Big( \frac 1{b_n} \nabla^\perp \psi_n^k\Big) = -\int_{\Omega_{n}} \div\Big[ \frac{1}{b_n} \chi^j\nabla \psi_n^k \Big] \; dx =- \int_{\Omega_{n}} \div\Big[ \frac{1}{b_n} \varphi_n^j\nabla \psi_n^k \Big] \; dx ,$$
for $\varphi_n^j$ defined in the previous subsection. Indeed, we can replace $\chi^j$ by $\varphi_n^j$ in \eqref{circulation} by density of $C^\infty_{c}(\Omega_{n})$ in $X_{b_n}$:  an argument already used in the proof of Proposition \ref{decomposition} (see \eqref{cancellation}).

Now, since $\{\varphi_n^k\}_{k = 1,\cdots,N}$ forms a basis (see Proposition \ref{prop basis}), we can write 
\begin{equation}\label{exp-psin} \psi_n^k = \sum_{j = 1}^N a_n^{(k,j)} \varphi_n^j.\end{equation}
Thus, by \eqref{def-psikn}, we have 
$$ \delta_{jk} = -\int_{\Omega_{n}} \frac{1}{b_n} \nabla \psi_n^k \cdot \nabla \varphi_n^j \; dx  = -\sum_{l=1}^N a_n^{(k,l)}  \int_{\Omega_{n}} \frac{1}{b_n} \nabla \varphi_n^l \cdot \nabla \varphi_n^j \; dx .$$
Let $A_n $ be the $N\times N$ matrix with components $a_n^{(j,k)}$ and $\Phi_n$ the matrix formed by $ \int_{\Omega_{n}} \frac{1}{b_n} \nabla \varphi_n^k \cdot \nabla \varphi_n^j \; dx.$ By Lemma \ref{lem-similifns}, $\Phi_n$ is well-defined and is uniformly bounded in $n$. We also let $A$ and $\Phi$ be the matrix obtained from $A_n$ and $\Phi_n$ by replacing $b_n$ by $b$ and $\varphi_n^j$ by $\varphi^j$. The above identity yields that $-I = A_n \Phi_n$ and Lemma \ref{lem-similifns} implies that $\Phi_n \to \Phi$. To get that $A_n \to A$ we need to prove that $\Phi$ is invertible. If $(\OM,b)$ is a smooth lake, then it is obvious because we also have $I=- A\Phi$. Concerning non-smooth lake, the invertible property comes from the positive capacity of islands (see \cite[Sub. 2.2]{GV_lac} for all details).

The expansion \eqref{exp-psin} then yields the following lemma. 

\begin{lemma}\label{lem-convpsikn} $b_n^{-1/2} \nabla \psi_n^k$ converges strongly in $L^2(D)$ to $b^{-1/2} \nabla \psi^k$, for each $k$. 
\end{lemma}

\subsection{Estimates of $\alpha_{n}^k$}

As the circulation is conserved $\gamma^{k}_n(v_{n})=\gamma^k_{n}$ (see Proposition \ref{circ const}), it is easy to get from the uniform bound of $\sqrt{b_{n}} \omega_{n}$ in $L^2$ (see Lemma \ref{lem-vorticity}), of $\sqrt{b_{n}}$ (see Definition \ref{def-ConvLake}), of $\varphi_{n}^k$  in $L^2$ (see Lemma \ref{lem-similifns}), that $\alpha_{n}^k$ is uniformly bounded in time and in $n$. From the boundedness, we deduce directly that
\[
\alpha_{n}^k \text{ converges weak-$\ast$ in }L^\infty(\R_+) \text{ to } \alpha^k(t)=\gamma^k+\int_{\Omega}b\omega \varphi^k.
\]

\subsection{Kernel part with Dirichlet condition} Let us next deal with the kernel part 
\begin{equation}\label{eq-kernel}
\div \Big[\frac1{b_n} \na \p^0_n\Big] = b_n \om_n, \quad \p^0_n\vert_{\pd  \Omega_n}=0.
\end{equation}

\begin{lemma}\label{lem-psi0} $\psi_n^0$ converges weakly-$\ast$ in $W^{1,\infty}(\R_+;H^1(D))$ to $\psi^0$, which is the solution of
\[\div \Big[\frac1{b} \na \p^0\Big] = b \om, \quad \p^0\vert_{\pd  \Omega}=0.\]
Furthermore, there holds the strong convergence  
$$ \frac{1}{\sqrt{b_n}}\nabla \psi_n^0 \to  \frac{1}{\sqrt b} \nabla \psi^0 \qquad \text{ strongly in }\quad  L^2((0,T)\times D) \text{ for any } T>0.$$
\end{lemma}
\begin{proof} Multiplying \eqref{eq-kernel} by $\psi_n^0$, we get 
\begin{equation}\label{id-psi0n} \int_{\Omega_n} \frac 1{b_n} |\nabla \psi_n^0|^2 \; dx  = -\int_{\Omega_n} b_n \omega_n \psi_n^0 \; dx \le \|\sqrt{b_n} \omega_n\|_{L^2} \|\sqrt{b_n}\psi_n^0 \|_{L^2} ,\end{equation}
in which $\|\sqrt{b_n} \omega_n\|_{L^2}$ is bounded thanks to Lemma \ref{lem-vorticity}. Using the Poincar\'e inequality on $D$ with Definition \ref{def-ConvLake}, we obtain that
\[
\|\sqrt{b_n}\psi_n^0 \|_{L^2(\Omega_{n})} \le \sqrt{M}\|\psi_n^0 \|_{L^2(\Omega_{n})} \le  c_{0} \sqrt{M}\| \nabla \psi_n^0 \|_{L^2(\Omega_{n})}  \le c_0 M  \|\frac 1{\sqrt{b_n}}\nabla \psi_n^0 \|_{L^2(\Omega_{n})},
\]
hence $\frac 1{\sqrt{b_n}}\nabla \psi_n^0$ and $\nabla \psi_n^0$ are uniformly bounded in $L^2(D)$, which implies that $\psi_n^0$ is uniformly bounded $H^1_0(D)$.

Putting together all the uniform bounds obtained in this section, we finally see that
\[
\sqrt{b_{n}} v^0_{n}=b_{n}^{-\frac12} \nabla \psi^0_{n} +\sum_{k=1}^N \alpha_{n}^k(0) b_{n}^{-\frac12} \nabla \psi^k_{n} \text{ is uniformly bounded in } L^2(D).
\]
It is now possible to state that $\sqrt{b_{n}} v_{n}$ is uniformly bounded in $L^\infty(\R_+;L^2(D))$ by the standard energy estimate, which is useful in the following estimate.

Similarly, $\partial_t \psi_n^0$ solves 
$$\div \Big[\frac1{b_n} \na \partial_t\p^0_n\Big] = \partial_t(b_n \om_n) = - \div (b_n v_n \omega_n), \quad \partial_t \p^0_n\vert_{\pd  \Omega_n}=0,$$
from which we obtain in the same way that
$$ \Big \| \frac 1{\sqrt{b_n}} \partial_t\nabla \psi_n^0\Big\|_{L^2(\Omega_{n})} \le \| \sqrt {b_n} v_n \|_{L^2} \| b_n \omega_n\|_{L^\infty}\lesssim1.$$
It follows that $ \frac 1{\sqrt{b_n}} \nabla \psi_n^0$ belongs to $W^{1,\infty}(\R_+;L^2(D))$ and $\psi_n^0 $ is in $W^{1,\infty}(\R_+; H^1_0(D))$. Consequently, up to some subsequence, there holds that
\begin{equation}\label{conv-weakpsi0}
 \frac{1}{\sqrt{b_n}}\nabla \psi_n^0 \rightharpoonup  \frac{1}{\sqrt b} \nabla \psi^0 \qquad \text{ weakly-}\ast \text{ in }\quad  L^\infty(\R_+; L^2(D)),
\end{equation}
and 
\[
 \psi_n^0 \to   \psi^0 \qquad \text{ weak-$*$ in }\quad  W^{1,\infty}(\R_+; H^1_{0}(D)) \text{ and  strongly in } C(\R_+; L^2(D)).
\]
By Mosco's convergence (see Proposition \ref{prop10}), it follows that $\psi^0 \in H_0^1(\Omega)$. Furthermore, it follows easily that $\psi^0$ solves 
\begin{equation}\label{eq-kernel0}
\div \Big[\frac1{b} \na \p^0\Big] = b \om, \quad \p^0\vert_{\pd  \Omega}=0,
\end{equation}
in the distributional sense. 

Next, by using the weak convergence of $\sqrt{b_n} \omega_n$ and strong convergence of $\sqrt{b_n}\psi_n^0$ in $L^2$, the identity \eqref{id-psi0n} then yields 
$$\int_{0}^T \int_{\Omega_n} \frac 1{b_n} |\nabla \psi_n^0|^2 \; dx  = -\int_{0}^T\int_{\Omega_n} b_n \omega_n \psi_n^0 \; dx  \to  - \int_{0}^T\int_{\Omega} b \omega \psi^0 \; dx  =  \int_{0}^T\int_{\Omega} \frac 1{b} |\nabla \psi^0|^2 \; dx, $$
in which the last identity follows from the equation \eqref{eq-kernel0}. Thus, the convergence \eqref{conv-weakpsi0} is indeed a strong convergence in $L^2((0,T)\times D)$. This proves the lemma.
\end{proof}

\subsection{Convergence of $\alpha_n^k$} In view of \eqref{BS-law}, we next study the convergence of $\alpha_n^k(t)$. We have already obtained a uniform bound in $L^\infty (\R_{+})$. Using the boundary condition $b_nv_n\cdot \nu_n=0$ in the last identity below, it follows that 
$$ \partial_t\alpha_n^k(t) = -  \int_{\Omega_n} \partial_t(b_n \omega_n)\varphi_n^k\; dx =  \int_{\Omega_n} \div(b_n v_n \omega_n)\varphi_n^k\; dx = - \int_{\Omega_n} b_n \omega_n \sqrt{b_n} v_n \cdot \frac{1}{\sqrt{b_n}}\nabla \varphi_n^k\; dx ,$$ 
which is again bounded by $L^2$ estimates. The strong convergence of $\alpha_n^k(t) $ to $\alpha^k(t)$ in $L^2((0,T))$ for any $T>0$ thus follows from this bound in $W^{1,\infty}(\R_+)$.

\subsection{Passing to the limit in the lake equation} It is now easy by \eqref{def-streamfn} and the expression \eqref{BS-law} to construct the limiting solution. Indeed, we recall from \eqref{BS-law} that $$\begin{aligned}
\psi_n (t,x) &= \psi_n^0(t,x) + \sum_{k=1}^N \alpha_n^k(t) \psi_n^k(x)
\end{aligned}
$$
with $\psi_n^0$ constructed as in \eqref{eq-kernel} and $\psi_n^k$ as in \eqref{def-psikn}. Lemmas \ref{lem-convpsikn} and \ref{lem-psi0} together with the convergence of $\alpha_n^k(t)$ then yield that the limiting function $\psi$ satisfies 
$$\psi (t,x)= \psi^0(t,x) + \sum_{k=1}^N \alpha^k(t) \psi^k(x).$$ 
We then introduce the limiting velocity through 
$$v: = \frac 1b \nabla^\perp \psi. $$
It follows clearly that $\sqrt{b_{n}}v_n \to \sqrt{b}v$ strongly in $L^2_{\loc}(\R_+; L^2(D))$. For any test function $\varphi\in C^\infty_{c}([0,\infty)\times \Omega)$, by the Hausdorff convergence, there exists $N_{\varphi}$ such that for any $n\geq N_{\varphi}$, $\varphi(t,\cdot)$ is compactly supported in $\Omega_{n}$ for all $t$. As $(v_{n},\omega_{n})$ is a global interior solution in the sense of Definition \ref{defi-weak-sol}, hence \eqref{eq-vort} holds for any $n\geq N_{\varphi}$, and for the limit. In addition, the divergence-free and boundary conditions follow at once from Lemma \ref{lem div} and our construction of the approximate solutions: $ \psi_n^k = \sum_{j = 1}^N a_n^{(k,j)} \varphi_n^j$ with $\varphi_n^k = \tilde \varphi_n^k + \chi^k$ (see \eqref{exp-psin}).

Therefore, the limit $v$ enjoys a Biot-Savart decomposition, and passing to the limit in the circulation definition \eqref{circulation} we obtain that the circulations of $v$ are conserved.

One notices that all the convergence results hold up to a subsequence extraction. However, if the limit lake is smooth, i.e. $(\partial\Omega,b)\in C^3\times C^3(\overline{\Omega})$ verifying assumptions (H1)-(H3), then we have proved that $(v,\omega)$ is a global weak interior solution in the vorticity formulation for the lake $(\Omega,b)$, with constant circulations, hence $(v,\omega)$ is also a global weak solution in the vorticity formulation (see Proposition \ref{prop A4}) and a global weak solution in the velocity formulation (see Proposition \ref{prop A3}). The uniqueness result implies that the whole sequence converges to the unique solution of the lake equations. This ends the proof of Theorem \ref{theo-main}.

\begin{remark}
In the previous proof, we never use that the islands are simply connected, hence we can relax this condition by assuming that $\Cc^i$ is a connected compact subset of $\Omega$. Indeed, in \cite[Proposition 1]{GV_lac}, it is proved that any connected compact set $\Cc^i$ can be approximated, in the Hausdorff topology, by smooth simply-connected compact set. Therefore, the case of a smooth simply-connected island which closes on itself (giving at the limit an annulus) is included in our analysis (see \cite[Section 5.1]{GV_lac} for pictures).
\end{remark}

\section{Non-smooth lakes}\label{sect non-smooth}
Let $(\Omega,b)$ be a lake satisfying (H1)-(H2). We assume that the $H^1$ capacity of all the islands is positive: ${\rm cap\ }(\Cc^k)>0$ for all $1\le k\le N$. Here, we assume no regularity on the boundary $\partial \Omega$. 

\bigskip

\noindent
{\em Domain approximation.} Without assuming any regularity on $\Omega$, we infer that $\Omega$ verifying (H1) is the Hausdorff limit of a sequence 
 $$\Omega_n \: := \:  \widetilde \Omega_n \: \setminus \: \left( \cup_{i=1}^k \overline{O_n^i}\right),$$
where  $\widetilde \Omega_n$ and $O_n^i$'s  are smooth Jordan domains, and such that  $\widetilde \Omega_n$, resp.   $\overline{O_n^i}$,   converges  in the Hausdorff sense to  $\widetilde \Omega$, resp.   $\Cc^i$. Such a property is a consequence of the Hausdorff topology and a proof can be found in \cite[Proposition 1]{GV_lac}. Moreover, therein, the sequence $\Omega_{n}$ can be constructed to be increasing thanks to the assumption that the obstacles $\Cc^{i}$ are simply connected\footnote{If $\Cc^{i}$ is a simply connected compact set, there exists a Riemann mapping $\Tc$ from $(\Cc^{i})^c$ to the exterior of the unit disk. Then,  $O_n^i:=(\Tc^{-1}(B(0,1+1/n)^c))^c$ is a smooth Jordan domain such that $\Cc^i \subset O_{n+1}^i \subset O_{n}^i$.}.

\bigskip

\noindent
{\em Bottom approximation.} We assume that $b$ is a bounded positive function on $\Omega$. It follows that there is a sequence $b_{n}\in C^\infty(\Omega_{n})$ with $M+1\geq b_{n}\geq \theta_{n}>0$ on $\Omega_{n}$ such that  $b_n$ converges strongly to $b$ in $L^p_{\loc}(\Omega)$ for any  $p\in [1,\infty)$.  
Indeed, we may define
\[
b_{n}:=(\rho_{n}*b)_{\vert \Omega_{n}}+\tfrac1n.
\]
Hence, the lake $(\Omega_{n},b_{n})$ is a smooth lake with a non-vanishing topography ($\theta_{n}=\frac1n$). Moreover, since for any compact set $K\subset \Omega$ there exists $\theta_{K}>0$ such that $b(x)\geq \theta_{K}$ on $K$, then there exists $n_{0}(K)$ such that for all $n\geq n_{0}(K)$ we have $b_{n}(x)\geq \theta_{K}/2$ on $K$. It then follows that the lake $(\Omega_{n},b_{n})$ converges to $(\Omega,b)$ in the sense of Definition \ref{def-ConvLake}.

Moreover, if $b\in W^{1,\infty}_{\loc}(\Omega)$ then our approximation $b_{n}$ convergences weakly to $b$ in $W^{1,\infty}_{\loc}(\Omega)$.

\bigskip

\noindent
{\em Initial data approximation.} For a function $u$ defined on a subset $U$ of $D$, we define $\underline{u}$ by $\underline{u}(x)=u(x)$ if $x\in U$ and $\underline{u}(x)=0$ if $x\in D\setminus U$. If $\omega^{0}\in L^\infty(\Omega)$ and $\gamma\in \R^N$ is given, then we consider $v_{n}^0$ such that
\[
\div (b_{n} v_{n}^0)=0, \quad \curl v_{n}^0 = b_{n} \underline{\om^0}\vert_{\OM_{n}}, \quad (b_{n} v_{n}^0)\cdot \nu \vert_{\pd \OM_{n}}=0,
\]
with its circulation around each $\overline{O_n^k}$ equal to $\g^k$, for all $k=1,\cdots,N$.

\bigskip

\noindent
{\em Existence result.} Similarly to the analysis of $(v_n,\omega_n)$ done in Section \ref{sect 3}, we get, up to extraction of a subsequence, that
\[ v_n \to v \text{ strongly in } L^2_{\loc}(\R_+; L^2(D)),\qquad   \om_n\rightharpoonup   \om \text{ weakly in } L^\infty(\R_+\times D) , \]
for some limiting pair $(v,\omega)$. It also follows that $(v,\omega)$ is a global weak interior solution in the vorticity formulation of the lake equations on the lake $(\OM, b)$ with initial vorticity $\omega^0$ and initial circulation $\gamma \in \R^N$. Furthermore, this constructed solution also enjoys a Biot-Savart decomposition, with constant circulations. For this part, we do not use any regularity on $b$.

To prove that $(v,\omega)$ is a global weak solution in the vorticity formulation, we take an arbitrary $\varphi\in C^\infty_{c}([0,\infty)\times \overline{\Omega})$ and verify \eqref{eq-vort} for $(v,\omega)$. Indeed, since $\Omega_n$ is increasing, $\varphi\vert_{\Omega_{n}}\in C^\infty_{c}([0,\infty)\times \overline{\Omega_{n}})$ is a test function for which \eqref{eq-vort} holds for $(v_{n},\omega_{n})$ (see Remark \ref{rem vort}). It is now easy to pass to the limit in \eqref{eq-vort}, which gives at once that  $(v,\omega)$ is a global weak solution in the vorticity formulation, even for test functions which are not constant on the boundary.

As for {\it interior} solution in the velocity formulation, we have to consider $b\in W^{1,\infty}_{\loc}(\Omega)$. In this case, we have stronger convergence of $b_{n}$ to $b$, which allows us to pass to the limit in the velocity equations \eqref{eq-vel}.

This completes the proof of Theorem \ref{theo-non-smooth}.

\bigskip

\noindent
{\em Remark on initial velocity.} If the initial data is given in terms of $v^{0}\in L^1_{\loc}(\Omega)$ such that $\omega^0:=\frac{\curl v^{0}}b\in L^\infty(\Omega)$, then the generalized circulation of a vector field $v$ around $\Cc^k$ is well defined:
\begin{equation*}
\g^k(v^0)= -\int_{\OM}\left(\na^\perp \chi^k\cdot v^0 + \chi^k \curl v^0\right)\, dx .
\end{equation*}
Hence we can consider $v_{n}^0$ such that
\[
\div (b_{n} v_{n}^0)=0, \quad \curl v_{n}^0 ={b_{n}} \underline{\tfrac{\curl (v^0)}b}\vert_{\OM_{n}}, \quad (b_{n} v_{n}^0)\cdot \hat{n} \vert_{\pd \OM_{n}}=0,
\]
with its circulation around each $\overline{O_n^k}$ equal to $\g^k(v^0)$, for all $k=1,\cdots,N$. Therefore, the previous compactness argument gives a solution with an initial velocity $\tilde v^0$ which has the same properties than $v^0$ (namely, same vorticity, circulations, and the same divergence and tangency condition). Nevertheless, it is not clear that $\tilde v^0=v^0$, even for $b$ lipschitz, because we need that the lake is smooth to apply Lemma \ref{1to1}.

\bigskip

\noindent
{\em Remark on uniqueness.} As written in the introduction, the uniqueness is not clear for non-smooth lake. To prove uniqueness in Section \ref{sect WP}, we need that the velocity belongs to $W^{1,p}$ for any $p<\infty$ with good bounds, which follows from the Calder\'on-Zygmund inequality. However, such an inequality is only true for smooth lake (i.e. $(\partial\Omega,b)\in C^3\times C^3(\overline{\Omega})$) and the first author shows in \cite{lac-uni} that the velocity for Euler equations blows up near an obtuse corner.
Therefore, the uniqueness result seems challenging for non-smooth domains. Even if the first author obtained a uniqueness result for the Euler equations adding some assumptions (namely, $\om^0$ is assumed to be compactly supported with definite sign, and $\OM$ is a simply connected bounded open set which is smooth except in a finite number of points), it is not clear how to adapt those techniques to the lake equations (e.g. to have an explicit formula for the Green kernel when the bottom is not flat).

\bigskip

\noindent
 {\bf Acknowledgements.} The authors are grateful to Didier Bresch for pointing us the interest of non-smooth lakes.
 
 The first author is partially supported  by the Project ``Instabilities in Hydrodynamics'' funded by Paris city hall (program ``Emergences'') and the Fondation Sciences Math\'ematiques de Paris and by the Project MathOc\'ean, grant ANR-08-BLAN-0301-01, financed by the Agence Nationale de la Recherche. Research of T.N. is supported in part by the NSF under grant DMS-1108821.  
 
\appendix 

\section{Equivalence of the various weak formulation} \label{sect equiv defi}

The goal of this section is to link together the various formulations of weak solutions to the lake equations, precisely between Definition \ref{defi-weak-velocity} and Definition \ref{defi-weak-sol}. First, it is obvious that
\begin{itemize}
\item a global weak solution of the velocity formulation is a global {\em interior} weak solution of the velocity formulation;
\item a global weak solution of the vorticity formulation is a global {\em interior} weak solution of the vorticity formulation;
\end{itemize}

Next, for any vector field $v$ such that $\div bv =0$, we can compute
\begin{equation}\label{algebra}
\div(bv\otimes v) = bv \cdot \nabla v = (bv)^\perp \curl v  + \frac{b}2 \nabla |v|^2.
\end{equation}
This equality is the key of the following propositions.

\begin{proposition}\label{prop A1}
Let $(\Omega,b)$ be a lake satisfying (H1)-(H2) with $b\in W^{1,\infty}_{\loc}(\Omega)$. Then a global {\em interior} weak solution of the velocity formulation is a global {\em interior} weak solution of the vorticity formulation.
\end{proposition}

\begin{proof}
Let us fix a test function $\varphi\in C^\infty_{c}([0,\infty)\times \OM)$, then $\F:=\nabla^\perp \varphi$ is divergence free and belongs to $C^\infty_{c}([0,\infty)\times \OM)$. As $v$ is a global interior weak solution of the velocity formulation, there holds
\begin{equation*}\begin{split}
0=&\int_0^\infty\int_{\OM} \Big[ \Phi_t \cdot v  + (bv\otimes v): \nabla\Bigl(\frac{\Phi}b \Bigl) \Big] \, dxdt+\int_{\OM}\Phi(0,x) \cdot v^0(x)\, dx\\
=&\int_0^\infty\int_{\OM} \Big[ \Phi_t \cdot v - \div(bv\otimes v)\cdot\Bigl(\frac{\Phi}b \Bigl) \Big] \, dxdt + \int_0^\infty \int_{\partial\OM} \Phi (v\otimes v) \nu \, d\tau dt +\int_{\OM}\Phi(0,x) \cdot v^0(x)\, dx.
\end{split}\end{equation*}
The boundary term vanishes because the support of test function does not intersect the boundary. Next, we use \eqref{algebra} and integrate by parts the linear terms to get
\begin{equation*}\begin{split}
0=&\int_0^\infty\int_{\OM} \Big[ \varphi_t \curl v + (\curl v) v^\perp \cdot \nabla^\perp \varphi  +\frac12 \nabla |v|^2  \cdot \nabla^\perp \varphi  \Big] \, dxdt +\int_{\OM}\varphi(0,x) \curl v^0(x)\, dx.
\end{split}\end{equation*}
Integrating by part the third integral and setting $\om:=b^{-1} \hbox{curl}(v)$, we then find
\[
0=\int_0^\infty\int_{\OM}\Big[ \varphi_t b\om + b\om v \cdot \nabla \varphi  \Big] \, dxdt+\int_{\OM}\varphi(0,x) b\om^0(x)\, dx,
\]
which is \eqref{eq-vort}. This ends the proof.
\end{proof}

The following confirms the inverse of Proposition \ref{prop A1} in the case the domain is simply connected.

\begin{proposition}\label{prop A2}
Let $(\Omega,b)$ be a lake satisfying (H1)-(H2), with $b\in W^{1,\infty}_{\loc}(\Omega)$ and with $N=0$, i.e. we assume that $\Omega$ is simply connected. Then a global interior weak solution of the vorticity formulation is a global interior weak solution of the velocity formulation.
\end{proposition}

\begin{proof}
Let us fix a divergence free test function $\F \in C^\infty_{c}([0,\infty)\times \OM)$, then there exists a stream function $\varphi$ such that $\F=\nabla^\perp \varphi$. As $\F$ is compactly supported, we infer that $\varphi$ is constant in a neighborhood of the boundary. If $\Omega$ is simply connected, there is only one connected component of $\partial \OM$, and as $\varphi$ can be chosen up to a constant, then we can consider $\varphi$ vanishing in the neighborhood of $\partial\OM$. The conclusion follows from the same computations as in the previous proposition.
\end{proof}

Concerning solutions up to the boundary, we need more regularity in order to justify \eqref{eq-vel} and the boundary terms in the integrations by parts. First, we show the following technical lemma which will be useful for the next proposition.

\begin{lemma}\label{DotProd}
Assume that $\Omega$ is a $C^3$-domain, that $v\in W^{1,p}(\Omega)$ for some $1\le p\le \infty$ satisfies $v\cdot\nu=0$ in $\partial\Omega$. Let $d(x)=\hbox{dist}(x,\partial\Omega)$ and let $\mathcal{N}$ be a neighborhood of $\partial\Omega$ where $d$ is $C^3$. Then
\begin{equation*}
v^\prime_d:=v\cdot\frac{\nabla d}{d}\in L^p(\mathcal{N}),\quad\hbox{ and }\quad \Vert v^\prime_d\Vert_{L^p(\mathcal{N})}\lesssim \Vert \nabla v\Vert_{L^p(\Omega)}
\end{equation*}
\end{lemma}

\begin{proof}
We extend $\nu$ into a vector field on $\mathcal{N}$ by setting $\nu(x)=-\nabla d(x)$.
For $x\in\mathcal{N}$, we introduce $\gamma_s(x)$ the solution of
\begin{equation*}
\frac{d}{ds}\gamma_s(x)=-\nabla d(\gamma_s(x)),\quad \gamma_0(x)=x
\end{equation*}
and let $\overline{x}=\lim_{s\to d(x)}\gamma_s(x)\in\partial\Omega$. Then, we simply remark that
\begin{equation*}
\begin{split}
v(x)&=v(\overline{x})-\int_0^{d(x)}\nu(\gamma_s(x))\cdot\nabla v(\gamma_s(x))ds\\
\nu(\gamma_s(x))&=\nu(x)=\nu(\overline{x})
\end{split}
\end{equation*}
and therefore, we see that
\begin{equation*}
v(x)\cdot\nu(x)=v(\overline{x})\cdot\nu(\overline{x})-d(x)\int_0^1\left[\nu\otimes\nu:\nabla v\right](\gamma_{sd(x)}(x))ds.
\end{equation*}
Since for any $0\le s\le 1$, the mapping $x\mapsto \gamma_{sd(x)}(x)$ has Jacobian uniformly bounded (this can be seen from the fact that $x\mapsto (d(x),\overline{x})$ has Jacobian uniformly bounded), we see that
\begin{equation*}
v^\prime_d(x)=-\int_0^1\left[\nu\otimes\nu:\nabla v\right](\gamma_{sd(x)}(x))ds
\end{equation*}
is bounded in $L^p$.

\end{proof}

Finally, let us give the equivalence between the formulations when the circulations around $\mathcal{C}^k$ for all $k$ are independent of the time.

\begin{proposition}\label{prop A3}
Let $(\Omega,b)$ be a lake verifying (H1)-(H3) and $(\partial \OM,b)\in C^3\times C^3(\overline{\Omega})$. Then a global weak solution of the vorticity formulation, whose the circulations are constant in time, is a global weak solution of the velocity formulation. Conversely a global weak solution of the velocity formulation, whose the circulations are constant in time, is also a global weak solution of the vorticity formulation.
\end{proposition}

\begin{proof}
In a smooth lake, then the global weak solution $v$ is more regular, namely thanks to the points i) and ii), the Calder\'on-Zygmund inequality (see Proposition \ref{prop CZ}) implies that $v$ belongs to $L^\infty(\mathbb{R}_+,W^{1,p}(\OM))$ for any $p\in (4,\infty)$ and that $v\cdot  \nu =0$ on $\partial\Omega$.

First, let us check that each term in \eqref{eq-vel} indeed makes sense. Note in particular that in the case of a boundary with vanishing topography, the term $b^{-1}\Phi$ can be unbounded. To fix this problem, we observe that
\begin{equation*}
(bv\otimes v): \nabla\Bigl(\frac{\Phi}b \Bigl)=v\otimes v:\nabla\Phi-\frac{v\cdot\nabla b}{b} v\cdot\Phi.
\end{equation*}
The first term creates no difficulty. For the second, we remark that given our assumption \eqref{BForm}, we have that, on $\mathcal{O}^k$,
\begin{equation*}
\frac{\nabla b(x)}{b(x)}=a_k\frac{\nabla d(x)}{d(x)}+\frac{\nabla c(x)}{c(x)}.
\end{equation*}
Once again, the second term does not create any problem, while for the first, we may use Lemma \ref{DotProd} to see that $v\cdot \frac{\nabla d(x)}{d(x)}$ in fact belongs to $L^\infty(L^p)$ for $p\in (4,\infty)$, which ends the justification of  \eqref{eq-vel}.

To prove the equivalence of the two formulations, we note that the fact that $\Phi \in C^\infty(\Omega)$ is tangent to the boundary implies that there exists $\varphi \in C^\infty(\Omega)$ constant on each connected component of $\partial \Omega$ such that $\Phi=\nabla^\perp\varphi$, and conversely $\varphi \in C^\infty(\Omega)$ with $\partial_{\tau} \varphi \vert_{\partial\Omega}=0$ implies that $\Phi:=\nabla^\perp\varphi$ is divergence free and tangent to the boundary. Therefore, it suffices to check that the boundary terms vanish in every integration by parts, following the proof in Proposition \ref{prop A1}:
\begin{itemize}
\item in the first integral, we have $\F(v\otimes v)\nu= (\F \cdot v)(v\cdot \nu)$ which is equal to zero because $v$ is tangent to the boundary (see Proposition \ref{prop CZ});
\item for the linear terms, we compute that
\begin{equation*}
\begin{split}
\int_0^{\infty}\int_\OM\Phi_t\cdot v\, dxdt+\int_\OM\Phi(0,x)\cdot v^0(x)\, dx
&=\int_{\OM}\left\{\int_0^\infty\nabla^\perp\varphi_t\cdot v \, dt+\nabla^\perp\varphi(0,x)\cdot v^0(x)\right\}\, dx\\
&=-\int_{\OM}\left\{\int_0^\infty\varphi_t\curl v \,dt+\varphi(0,x)\curl v^0(x)\right\}\, dx\\
&\quad+\int_{\partial\OM}\left\{\int_0^\infty\varphi_t (v\cdot\tau)dt+\varphi(0,x)(v^0(x)\cdot\tau)\right\}d\sigma\\
\end{split}
\end{equation*}
Thanks to the regularity of $v$, the notions of weak circulation and classical circulation are equivalent: using the fact that $\varphi(t,\cdot)$ is constant on each connected component of the boundary and the conservation of the circulations, we can rewrite the last integral as
\begin{equation*}
\begin{split}
\int_{\partial\OM}\left\{\int_0^\infty\varphi_t (v\cdot\tau)dt+\varphi(0,x)(v^0\cdot\tau)\right\}d\sigma&=\sum_{k=1}^N\gamma_k(v)\left[\int_0^\infty\partial_t\varphi_{|\mathcal{C}^k}dt+\varphi(0,\cdot)_{|\mathcal{C}^k}\right]=0.
\end{split}
\end{equation*}
In this computation, we have assumed that $\varphi \vert_{\partial \widetilde\Omega}\equiv 0$, which is the general convention (because we always consider $\varphi$ up to a constant), however we could also state that the circulation on $\partial \widetilde\Omega$ is constant because the Stokes formula gives
 $\oint_{\partial \widetilde\Omega} v(x)\cdot\tau d\sigma = \sum_{k=1}^N\gamma_k(v) + \int_{\Omega} \curl v$, where the last integral is independent of time from the transport nature \eqref{transport}.

\item in the last integration by parts, the boundary term is $\int_{0}^\infty \int_{\partial \OM}|v|^2 ( \nu \cdot \nabla^\perp \varphi)\, dxdt$, which is equal to zero because we assume that $\partial_{\tau} \varphi\equiv 0$.
\end{itemize}
This ends the proof.
\end{proof}
We note in the last two bullets of the previous proof why we have assumed that $\F$ is tangent to the boundary in the sense of Definition \ref{defi-weak-velocity} and that $\varphi$ is constant on the boundary in Definition \ref{defi-weak-sol}.

Finally, for smooth lakes, we prove that it is sufficient to prove existence of an interior solution.
\begin{proposition}\label{prop A4}
Let $(\Omega,b)$ be a lake verifying (H1)-(H3) and $(\partial \OM,b)\in C^3\times C^3(\overline{\Omega})$. Then a global interior weak solution of the vorticity formulation is a global weak solution of the vorticity formulation.
\end{proposition}

\begin{proof}
Let $\varphi\in C^\infty_{c}([0,\infty)\times \overline{\Omega})$ and let
\begin{equation*}
\eta_\varepsilon(x)=\eta(\varepsilon^{-1}d(x)),\quad\eta_\varepsilon\in C^3(\Omega)
\end{equation*}
where $\eta\in C^\infty_c(\mathbb{R})$ is equal to one in $[-1/2,1/2]$ and vanishes outside of $[-1,1]$. As in the previous proof, we infer by Proposition \ref{prop CZ} and Lemma \ref{DotProd} that
\begin{equation*}
\Vert v\cdot\nabla\eta_\varepsilon\Vert_{L^p(\Omega)}\lesssim \Vert v\cdot d^{-1}\nabla d\Vert_{L^p(\partial\Omega_{\varepsilon})}=\mathcal{O}(\varepsilon^{1/(2p)}),
\end{equation*}
where $\partial\Omega_{\varepsilon}$ is a $\varepsilon$-neighborhood of $\partial\Omega$ as in \eqref{Neighborhood}. Again, this yields that $\| v\cdot \nabla \eta_{\varepsilon} \|_{L^p}\to 0$ when $\varepsilon\to 0$, for any $p\in [1,\infty)$.

Since $\eta_{\varepsilon} \varphi$ is now known to be $C^\infty_{c}([0,\infty)\times\Omega)$, the identity \eqref{eq-vort} holds for $\varphi$ replaced by $\eta_\varepsilon\varphi$
\begin{equation*}\begin{split}
 \int_0^\infty\int_{\OM} \Big[ \eta_{\varepsilon} \f_t b \om +  \f (\nabla \eta_{\varepsilon} \cdot v) b \om +  \eta_{\varepsilon}  (\nabla \f \cdot v) b \om \Big] \, dxdt
 + \int_{\R^2} \eta_{\varepsilon}\f(0,x)b \om^0(x)\, dx&=0.
\end{split}\end{equation*}
Thanks to the convergence of $v\cdot \nabla \eta_{\varepsilon}$ and the estimates on $v$ and $\omega$, we are now able to pass to the limit $\varepsilon \to 0$. This proves that $(v,\om)$ is a global weak solutions of the vorticity formulation, even for test functions in $C^\infty_c([0,\infty)\times \overline{\Omega})$.
\end{proof}

\begin{remark}\label{rem vort} We note in the previous proof that we do not need the condition ``$\varphi$ constant on the boundary''. Therefore, for smooth lake, we could avoid this condition in Definition \ref{defi-weak-sol}: indeed, a global weak solution in the sense of this definition (with the condition $\partial_{\tau}\varphi\vert_{\partial\Omega}\equiv 0$) is a global weak interior solution, and the previous proposition states that it is a global weak solution for any test functions in $C^\infty_c([0,\infty)\times \overline{\Omega})$ (without boundary condition).

We can also prove the following for any $\varphi\in C^\infty_c([0,\infty)\times \overline{\Omega})$:
\begin{equation*}
 \int_0^\infty\int_{\OM} \f_t \om \, dxdt + \int_0^\infty \int_{\OM}    (\nabla \f \cdot v)  \om \, dxdt + \int_{\R^2} \f(0,x) \om^0(x)\, dx=0.
\end{equation*}
Indeed, by assumption (H2), we can divide by $b$ if $\varphi\in C^\infty_{c}([0,\infty)\times \Omega)$, and thus passing to the limit holds in the same manner.
\end{remark}

\section{$\gamma$-convergence of open sets}  \label{app_gammaconv}
Let $D$ be a bounded open set.  Let  $(\Omega_n)_{n \in \N}$ be  a sequence of open sets  included in  $D$. One says that  $(\Omega_n)_{n \in \N}$ $\gamma$-converges to $\Omega \subset D$ if for any $f \in H^{-1}(D)$, the sequence of  solutions $\p_n \in H^1_0(\Omega_n)$ of
\begin{equation*}
 -\Delta \p_n = f \: \mbox{ in } \:  \Omega_n, \quad \p_n\vert_{\partial  \Omega_n} = 0,\quad\p_n\equiv0\mbox{ on }D\setminus\Omega_n
 \end{equation*}
converges in $H^1_0(D)$  to the solution $\p \in H^1_0(\Omega)$ of
$$ -\Delta \p = f \: \mbox{ in } \:  \Omega, \quad \p\vert_{\partial  \Omega} = 0.$$

In this definition, $H^1_0(\Omega)$ and $H^1_0(\Omega_n)$ are seen as subsets of $H^1_0(D)$, through extension by zero. In a dual way, $H^{-1}(D)$ is seen as a subset of 
$H^{-1}(\Omega_n)$ and $H^{-1}(\Omega)$. As for the Hausdorff convergence of open sets, the definition of $\gamma$-convergence does not depend on the choice of the confining set $D$.

\medskip
The notion of $\gamma$-convergence is extensively discussed in \cite{henrot}. The basic example of $\gamma$-convergence is given by  increasing sequences: 
\begin{proposition}  \label{prop8}
If $(\Omega_n)_{n \in \N}$ is an increasing sequence in $D$, it $\gamma$-converges to $\Omega \: = \:  \cup \, \Omega_n$. More generally, if  $(\Omega_n)_{n \in \N}$ is included in $\Omega$ and converges to $\Omega$ in the Hausdorff sense, then it $\gamma$-converges to $\Omega$. 
\end{proposition}
In general, Hausdorff converging sequences are not $\gamma$-converging. We refer to \cite{henrot} for counterexamples, with domains $\Omega_n$ that have  more and more holes as $n$ goes to infinity. This kind of  counterexamples, reminiscent of homogenization problems, is the only one in dimension 2, as proved by Sverak \cite{sverak}:
\begin{proposition} \label{prop9}
Let  $(\Omega_n)_{n \in \N}$ be a sequence of open sets in $\R^2$, included in $D$.  Assume that the number of connected components of  $D \setminus \Omega_n$ is bounded uniformly in $n$. If  $(\Omega_n)_{n \in \N}$ converges in the Hausdorff sense to $\Omega$, it $\gamma$-converges to $\Omega$.  
\end{proposition}  
This result is a crucial ingredient  in the convergence proofs. 

\medskip
One can characterize the $\gamma$-convergence in terms of the Mosco-convergence of $H^1_0(\Omega_n)$ to $H^1_0(\Omega)$. Namely: 
\begin{proposition} \label{prop10}
 $(\Omega_n)_{n \in \N}$ $\gamma$-converges to $\Omega$ if and only if the following two conditions are satisfied: 
 \begin{enumerate}
 \item For all $\p \in H^1_0(\Omega)$, there exists a sequence $(\p_n)_{n \in \N}$ in $H^1_0(\Omega_n)$ that converges strongly to $\p$. 
 \item For any sequence $(\p_n)_{n \in \N}$ with $\p_n$ in $H^1_0(\Omega_n)$, weakly converging to $\p$ in $H^1_0(D)$, $\p \in H^1_0(\Omega)$. 
 \end{enumerate}
 \end{proposition}
One can also  characterize  $\gamma$-convergence with capacity, see \cite[Proposition 3.5.5 page 114]{henrot}. Let us finally  mention that the notion of $\gamma$-convergence of open sets is related to the more standard $\Gamma$-convergence of Di Giorgi. Loosely speaking, $\Omega_n$ $\gamma$-converges to $\Omega$ if the corresponding Dirichlet energy functional $J_{\Omega_n}$ $\Gamma$-converges to $J_\Omega$: see \cite[section 7.1.1]{henrot} for all details.

\end{document}